\documentclass[12pt,reqno]{amsart}
\usepackage{amsmath,amssymb,amsfonts,amscd,latexsym,amsthm,mathrsfs,verbatim,comment,color,soul}
\usepackage{wasysym}
\usepackage{enumerate}
\usepackage{graphicx}
\usepackage{tikz}
\usetikzlibrary{calc,decorations.markings}\usetikzlibrary{intersections}
\usepackage{hyperref}
\newtheorem{prop}{Proposition}
\newtheorem{corollary}{Corollary}
\newtheorem{theorem}{Theorem}
\theoremstyle{remark}
\newtheorem{remark}{\bf Remark}

\newtheorem{definition}{\bf Definition}
\newtheorem{conjecture}{\bf Conjecture}
\let\p\partial

\renewcommand{\d}{{\mathrm d}}

\newcommand{\Jac}{\operatorname{Jac}}
\newcommand{\Braid}{\operatorname{Braid}}
\renewcommand{\d}{{\mathrm d}}
\newcommand{\DE}{\operatorname{DE}}
\newcommand{\Z}{\mathbb{Z}}
\newcommand{\Q}{\mathbb{Q}}
\newcommand{\C}{\mathbb{C}}
\newcommand*\circled[1]{\tikz[baseline=(char.base)]{
  \node[shape=circle,draw,inner sep=1pt] (char) {#1};}}

\newcommand\miprod{\mathop{\circled{\tiny P}}}

\begin{document}

\title[A generating function of the squares of~Legendre~polynomials]{A hyperelliptic
saga on a generating function of~the~squares~of~Legendre~polynomials}

\author{Mark van Hoeij}
\address{M.v.H.: Mathematics Department, Florida State University, Tallahassee, FL 32306, U.S.A.}
\urladdr{www.math.fsu.edu/~hoeij/}

\author{Duco van Straten}
\address{D.v.S.: Fachbereich 08, AG Algebraische Geometrie, Johannes Gutenberg-Univer\-sit\"at, D-55099 Mainz, Germany}
\urladdr{www.agtz.mathematik.uni-mainz.de/algebraische-geometrie/van-straten/}

\author{Wadim Zudilin}
\address{W.Z.: Department of Mathematics, IMAPP, Radboud University, PO Box 9010, 6500~GL Nijmegen, Netherlands}
\urladdr{www.math.ru.nl/~wzudilin/}

\thanks{The first author (M.v.H.) was supported by NSF grant 2007959.
The second author (D.v.S.) acknowledges support by the Deutsche Forschungsgemeinschaft (DFG) through the Collaborative Research Center TRR 326 \textit{Geometry and Arithmetic of Uniformized Structures}, project number 444845124.}

\date{17 January 2024.  \emph{Revised}: 28 May 2024}

\subjclass[2010]{11F99, 11Y60, 14H45, 14Q05, 33C20, 33E30, 34M35}
\keywords{Picard--Fuchs differential operator; integrality; hyperelliptic curve; Humbert's equation; Legendre polynomial}

\begin{abstract}
We decompose the generating function
$\sum_{n=0}^\infty\binom{2n}nP_n(y)^2z^n$
of the squares of Legendre polynomials
as a product of periods of hyperelliptic curves.
These periods satisfy a family of \emph{second} order differential equations. This is highly unusual since \emph{four} is the expected order for genus~2.
These second order equations are arithmetic and yet, surprisingly, their monodromy group is dense in $\operatorname{SL}_2(\mathbb{R})$.
This suggests that they cannot be solved in terms of hypergeometric functions,
which is novel for arithmetic second order differential equations that are \emph{defined over} $\mathbb{Q}$,
and also novel for a \emph{family} of such equations.
We complement our analysis with a recipe for constructing similar examples.
\end{abstract}

\maketitle
%==================================================

\section{Introduction}
\label{sec}

It would be not mistaken to call Legendre polynomials
\[
P_n(y)=\frac1{2^nn!}\,\frac{\d^n}{\d y^n}\big((y^2-1)^n\big)
=\sum_{k=0}^n\binom nk\binom{n+k}k\bigg(\frac{y-1}2\bigg)^k
\]
the most classical \emph{special} polynomials.
They can also be given by a recurrence equation
\begin{equation}
\begin{gathered}
(n+1)P_{n+1}(y)-(2n+1)yP_n(y)+nP_{n-1}(y)=0 \quad\text{for}\; n\ge0,
\\
\text{with}\; P_0(y)=1 \;\text{and}\; P_{-1}(y)=0,
\end{gathered}
\label{Leg-rec}
\end{equation}
or by a generating function
\begin{equation}
F_{P}(y,z)=\sum_{n=0}^\infty P_n(y)z^n=\frac1{\sqrt{1-2yz+z^2}}.
\label{eqLeg}
\end{equation}
Though many features of Legendre polynomials appear in more general families of special polynomials,
there are still surprises kept by and open questions for this classic.

Legendre polynomials can also be expressed with Euler--Gauss hypergeometric functions
\[
{}_2F_1\biggl(\begin{matrix} a, \, b \\ c \end{matrix}\biggm|x\biggr)
=\mbox{}_2F_1(a, b; c \mid x)
=\sum_{n=0}^\infty\frac{(a)_n(b)_n}{n!\,(c)_n}\,x^n,
\]
where $(a)_n =\Gamma(a+n)/\Gamma(a) =\prod_{j=0}^{n-1}(a+j)$ is Pochhammer's symbol.
This is true for the polynomials themselves
\[
P_n(y)={}_2F_1\biggl(\begin{matrix} -n, \, n+1 \\ 1 \end{matrix}\biggm| \frac{1-y}2 \biggr)
=\biggl(\frac{y+1}2\biggr)^n{}_2F_1\biggl(\begin{matrix} -n, \, -n \\ 1 \end{matrix}\biggm| \frac{y-1}{y+1} \biggr),
\]
as well as for the generating function of their squares
\begin{equation}
F_{P^2}(y,z)=\sum_{n=0}^\infty P_n(y)^2z^n
=\frac1{\sqrt{1-2y^2z+z^2}} \cdot {}_2F_1\biggl(\begin{matrix} \frac14, \, \frac34 \\ 1 \end{matrix}\biggm|
\frac{4(1-y^2)^2z^2}{(1-2y^2z+z^2)^2}\biggr)
\label{Wan}
\end{equation}
extracted from \cite{Zu14} (for which another `more radical' formula was given by W.\,N.~Bailey in the 1940s).
Some recent arithmetic questions, which we briefly outline below, suggested the existence of `explicit' expressions for the twist
\begin{equation}
	\tilde{F}_{P^2}(y,z)=\sum_{n=0}^\infty\binom{2n}nP_n(y)^2z^n
\label{orig}
\end{equation}
and for the generating function of the cubes
\begin{equation}
	F_{P^3}(y,z)=\sum_{n=0}^\infty P_n(y)^3 z^n.
\label{cub}
\end{equation}
In spite of considerable efforts, no such formulas had been found.

Apart from \eqref{Wan}, other reasons to expect existence of $_2F_1$ type expressions for $\tilde{F}_{P^2}(y,z)$ and $F_{P^3}(y,z)$
are based on similarities the two share with one of Appell's bivariate hypergeometric functions,
\begin{equation}
\mathrm F_4\biggl(\begin{matrix} a, \, b \\ c_1, c_2 \end{matrix}\biggm| z , w\biggr)
=\sum_{m,n\ge0}\frac{(a)_{m+n}(b)_{m+n}}{m!\,n!\,(c_1)_m(c_2)_n}\,z^mw^n.
\label{App}
\end{equation}
For the latter, we have the celebrated Barnes--Bailey identity  \cite{Ba33,Ba34,Ba54} (another classic!)
\begin{equation}
% F_4(a,b;\,c,a+b-c+1\mid x(1-y),y(1-x)) =
\mathrm F_4\biggl(\begin{matrix} a, \, b \, \\ \, c_1, \, c_2\ \end{matrix}\biggm|  x(1-y),y(1-x) \biggr)
={}_2F_1\biggl(\begin{matrix} a, \, b \\ c_1 \end{matrix}\biggm|x\biggr)
\,{}_2F_1\biggl(\begin{matrix} a, \, b \\  c_2 \end{matrix}\biggm|y\biggr)
\label{F4}
\end{equation}
when $c_1+c_2 = a+b+1$ (see \cite[Section~3]{Zu19}).
The functions \eqref{orig}, \eqref{cub} and \eqref{App} (for special choices of $a$, $b$ and $c_1=c_2=1$) show up in an arithmetic context of formulae for $\pi$,
which were experimentally discovered by Zhi-Wei Sun \cite{Su12}.
Resolutions for related conjectures from \cite{Su12} make use of the decomposition~\eqref{F4} and its numerous generalisations\,---\,see~\cite{CWZ13,CWZ18,RS13,Wa14,WZ12,WY22}.
At the same time this methodology does not seem to extend to examples from \cite{Su12} that involve $\tilde{F}_{P^2}(y,z)$ and $F_{P^3}(y,z)$.
Proofs of some formulae for $\pi$ with $\tilde{F}_{P^2}(y,z)$ are given in \cite{Zu14}; the crucial ingredient of those is a formula for \eqref{orig} in a special case $z=4(y^2-1)/(y^2+3)^2$.

For all these reasons,
our initial goal was a general formula for $\tilde{F}_{P^2}(y,z)$ with a product like \eqref{F4}.
And while we did find a product whose factors satisfy second order equations, they did not turn out to be ${}_2F_1$ expressible. Instead of
elliptic integrals (very common for arithmetic second order equations) for the factors, we were surprised to witness {\em hyperelliptic} integrals.

\begin{theorem} \label{main}
The following formula holds\textup:
\begin{equation}
	\tilde{F}_{P^2}(y, z)  =  w \, I_+(4z ,w^2)I_-(4z,w^2)
	 \label{decoup0}
\end{equation}
where
\[
	w = \sqrt{ (1 + 4z)^2 - 16 y^2 z} \, + \,  4y\sqrt{-z} 
\]
and
\begin{equation}
I_{\pm}(u, x)=\frac1\pi\int_0^1
\frac{1 - uv \pm  v  \sqrt{2u^2 - 2u}}
{\sqrt{\smash[b]{v \, (1-v) \, \big( (1-v)(1-u^2v)(1+uv)^2 + x \, v(1-uv)^2 \big)}}} \, \d v \,.
\label{eqI}
\end{equation}
\end{theorem}

\begin{proof}
It suffices to compute a differential equation
for both sides of \eqref{decoup0} and then compare initial conditions.
(The equation for $I_{\pm}(u,x)$ is in Section~\ref{RM-check}.)
\end{proof}

While this proof is technically valid, two issues remain.
First, it does not explain \emph{how the formula was found}\,---\,and this was not easy!\,---\,the details
are in Section~\ref{discovery} (which also serves as a more detailed proof).
Second, and more important, the proof ignores the fascinating properties of the objects involved.
The geometry behind the second order linear differential equation satisfied by $I_{\pm}(u,x)$ is quite remarkable.
The integral in \eqref{eqI} represents periods of hyperelliptic curves, of genus~2, for which it is highly unusual to satisfy \emph{second} order linear differential equations since \emph{four} is the expected order.
These equations are arithmetic%
\footnote{Here, arithmeticity of a differential equation refers to the existence of a globally bounded solution, that is, of a power series solution $f(x)=\sum_{n=0}^\infty a_n x^n\in1+x\mathbb Q[[x]]$ such that $f(c x)\in\mathbb Z[[x]]$ for some non-zero constant $c$.}
and yet, their monodromy group is {\em dense} in $\operatorname{SL}_2(\mathbb{R})$. This suggests that they cannot be solved in terms of $_2F_1$ hypergeometric functions,
which is novel for an arithmetic second order equation that is defined over $\mathbb{Q}$.
Indeed, in the OEIS \cite{OEIS}
there are no examples yet\footnote{At the time of writing, but this can now be changed by uploading an instance of Theorem~\ref{th1} or Remark~\ref{Remark2}!}
of $a_n$ in $\mathbb Z$ with $f(x) = \sum_{n=0}^\infty a_n x^n$
holonomic of order~2, a positive radius of convergence, where $f(x)$ is not expressible in terms of
algebraic functions, $\exp$, $\log$, ${}_2F_1$, etc.
This is despite the fact that the database \cite{OEIS} records {hundreds} of non-trivial examples that are arithmetic and differential order~2.

Being defined over $\mathbb{Q}$ and having a dense monodromy group (Theorem~\ref{monodromy-dense})
are not the only novelties in our arithmetic non-${}_2F_1$ example.  Also novel is that
it is not an isolated equation, but rather a family of equations, in which the 4 non-apparent singularities can move freely.\footnote{This requires $4-3 = 1$ parameter
since one can freely move 3 singularities with a M\"obius transformation.}
This discovery of a new class of arithmetic second order differential equations led to a considerable expansion of what was originally planned to be an ordinary mathematics paper finishing about the place where Section~\ref{discovery} ends.
The current exposition includes a discussion of the monodromy of the differential operator encountered and its explicit connection with the geometry of genus~2 curves with a specific feature\,---\,real multiplication. This is done in Sections~\ref{RM-check} and~\ref{monodromy}, while Section~\ref{dwork} compares our example to prior counter examples to a conjecture of Dwork.

For the generating function \eqref{cub} we give an identity of different sort in Section~\ref{cube}.
Perhaps a decomposition of \eqref{cub} similar to the one in Theorem~\ref{main} exists, but over an elliptic extension of $\mathbb Q(z)$.
Finally, Section~\ref{final}
%% shows how to construct similar examples, and directions for further exploration of the novelty.
summarizes directions for further exploration of the novelty.

\begin{remark}[analytic aspects of formula \eqref{decoup0}]
\label{Remark1}
For fixed $y \in \mathbb{C}$, the series $\tilde{F}_{P^2}(y,z) \in \mathbb{C}[[z]]$ has radius of convergence
\[
R_y = \frac{\min( | y \pm \sqrt{y^2-1} | )^2}{4}.
\]
We have $R_y > 0$ and $R_y \leq 1/4$ for all $y \in \mathbb{C}$, with equality if and only if $y$ is a real number in the interval $[-1,1]$.
Theorem~\ref{main} holds for all $y ,z \in \mathbb{R}$ with $|z| < R_y$.  It also holds for $y, z \in \mathbb{C}$ if $|z|$ is
small enough to ensure that the real parts of $(1 + 4z)^2 - 16 y^2 z =  1 + \mathcal{O}(z)$ and $(1-v)(1-u^2v)(1+uv)^2 + x v(1-uv)^2 = 1 + \mathcal{O}(z^{1/2})$ are positive for any $v \in [0,1]$
(here $u = 4z$ and $x = w^2$). This restriction on $|z|$ ensures that two of the square roots in Theorem~\ref{main} do not suddenly change sign
(the signs of $\sqrt{-z}$ and $\sqrt{2u^2 - 2u}$ do not affect the validity of~\eqref{decoup0}).
\end{remark}

\section{How the formula was found} \label{discovery}

Throughout this section we write $F(y,z)$ for $\tilde{F}_{P^2}(y,z)$.
This function satisfies the following differential system of holomorphic rank~4 % , which is generated by the equations
\begin{align*}
z \bigg(y^2-2z-\frac12 \bigg)\frac{\p^2F}{\p y\,\p z}
+\frac{y}2(1-y^2)\frac{\p^2F}{\p y^2}
-(y^2+z)\frac{\p F}{\p y}+yz\frac{\p F}{\p z}
&=0,
\\
F + \bigg(2z^2-\frac{z}2\bigg)\frac{\p^2F}{\p z^2} + \bigg(5z-\frac12\bigg)\frac{\p F}{\p z}
+\bigg(y-\frac1y\bigg)\bigg( \bigg(z+\frac14\bigg)\frac{\p^2F}{\p y\,\p z}
+\frac12 \frac{\p F}{\p y} \bigg) & = 0
\end{align*}
obtained from \cite[eq.~(1.4)]{Zu14}:
$$
F(y,z)=\sum_{n=0}^\infty\binom{2n}nz^n\sum_{k=0}^n\binom nk\binom{n+k}n\binom{2k}k\bigg(\frac{y^2-1}4\bigg)^k
$$
with creative telescoping.
After eliminating derivatives with respect to one variable we obtain fourth order linear differential equations
$\DE_y\in \mathbb{Q}(y,z)[\d / \d y]$ and $\DE_z\in \mathbb{Q}(y,z)[\d / \d z]$.
They are too large to print here but are available on our website~\cite{supp}.
One can also construct $\DE_z$ by using \texttt{Maple}'s \texttt{gfun} to convert a recurrence for $\binom{2n}n P_n(y)^2$ to a differential equation.

\begin{definition}\label{def1}
For a differential operator $L$,
let $V(L)$ denote the solution space of $L$ in a universal extension.
The \emph{least common left multiple}
	% $L = L_1 \miadd L_2$, also appearing in the literature as
$\operatorname{LCLM}(L_1, L_2)$
is the lowest order differential equation with $y_1 + y_2 \in V(L)$ for any $y_1 \in V(L_1)$, $y_2 \in V(L_2)$.
The \emph{symmetric product}  $L_1 \miprod L_2$  is defined similarly, just replace $y_1 + y_2$ by $y_1 \cdot y_2$.
\end{definition}

\subsection{Decomposing $\DE_z$} 
The first step to solving a differential equation is to check if it factors \cite{vH96}, however, $\DE_y$ and $\DE_z$ are irreducible.
The next step is the algorithm from \cite{vH02} which detects if a fourth order equation
is a symmetric product. It finds
$\DE_z = L^z_{+} \miprod L^z_{-}$, where
\begin{align}
L^z_{\pm}
&:= z(1-4z)(1+4z)^2((1+4z)^2-16y^2z)\,\frac{\d^2}{\d z^2}
\nonumber\\[-3pt] &\quad
   + (8z(48z^2+8z-3)y^2+(1-32z^2)(1+4z)^2)(1+4z)\,\frac{\d}{\d z}
\nonumber\\ &\quad
   + (64z^3+80z^2+4z-1)y^2 - 3z(1+4z)^3
\nonumber\\ &\quad
   \pm (1-8z)y \sqrt{2(1-4z)((1+4z)^2-16zy^2)}.
\label{L1}
\end{align}
For several reasons mentioned in the introduction, we very much expected $L^z_{+} $ and $L^z_{-} $ to have $_2F_1$ type solutions.
Our search for $_2F_1$ solutions was
time consuming because the algorithms for this are only implemented for rational function coefficients.
(The square root in $L^z_{\pm}$ has $z$-degree 3, so it cannot be rationalized.)

\subsection{Decomposing $\DE_y$}  \label{DEy}
Next we applied algorithm \cite{vH02} to $\DE_y$
and found it to be a symmetric product as well, namely $L^y_{+} \miprod L^y_{-}$ where
\begin{align}
L^y_{\pm} &= 2(1-y^2)(2y^2-4z-1)^2((1+4z)^2-16zy^2)\,\frac{\d^2}{\d y^2}
\nonumber\\[-8pt] & %\quad
+4y(2y^2-4z-1)(32zy^4 - (1+4z)((1+28z)y^2 - 4z(3+4z)))\,\frac{\d}{\d y}
\nonumber\\ & %\quad
+ 16(2y^2- 14z - 3)zy^4 + (1+4z)(1+32z+80z^2)y^2 - (1+4z)^2(1+4z+8z^2)  \nonumber\\ & %\quad
\pm (2(1+8z)y^2 -  4(1+4z)(1-z))y\sqrt{2(1-4z)((1+4z)^2 - 16zy^2)}.
\label{L2}
\end{align}
The polynomial in the square root has $y$-degree 2.
So this time we can rationalize it with a substitution. We choose
\begin{equation} \label{subs_xt}
y = \left(x + (t^2-1)^2\right) \sqrt{z/x}, \quad
z=\frac1{4(1-2t^2)}.
\end{equation}
Substituting \eqref{subs_xt} in \eqref{L2} produces equations $\tilde{L}^x_{\pm}$.
The factor $\sqrt{z/x}$ in $y$ does not cause a square root in $\tilde{L}^x_{\pm}$ because $F(y,z) = F(-y,z)$.

If a differential equation has non-zero exponents $e_1 \leq e_2$ at a point $p$,
then dividing its solutions by $(x-p)^{e_1}$ creates an exponent 0 at that point, which usually reduces the size of the equation.
Applying this idea to $\tilde{L}^x_{+}$ we divide its solutions by
\begin{equation}
\sqrt[4]{ x \frac{ x + (t-1)(t+1)^3 }{ x + (t+1)(t-1)^3 } }
% SWAPPED SIGN
\label{algfactor}
\end{equation}
which results in
\begin{align*}
&
L_{+}^x = \frac{\d^2}{\d x^2} + \left( \frac1x - \frac1{x+(t+1)(t-1)^3} + \frac{2(x+t^4+2t^2-1)}{x^2+(2t^4+4t^2-2)x +  (t^2-1)^4} \right)\frac{\d}{\d x}
\\ &\;
+ \frac{4x^2+(8t^4-16t^3-4t^2+20t-9)x  + (4t^4-8t^3+12t^2+4t-5)(t+1)(t-1)^3 }{16x (x^2+(2t^4+4t^2-2)x+(t^2-1)^4)  (x+(t+1)(t-1)^3)}.
\end{align*}
Repeating this for $L^y_{-}$ produces $L_{-}^{x} = L_{+}^{x} \rvert_{t \mapsto -t}$.  Then
\[
	\DE_y = L_1^x \miprod L_{+}^{x} \miprod L_{-}^{x}
\]
where \eqref{subs_xt} expresses $y,z$ in terms of $x,t$ and where $L_1^x = \frac{\d}{\d x} - \frac1{2x}$
whose solution $\sqrt{x}$ is the product of \eqref{algfactor} for $t$ and $-t$.

Let $F_{+}^{x} = \sum_{n=0}^{\infty} a_n x^n \in \mathbb{Q}(t)[[ x ]]$ be the unique monic ($a_0 = 1$)  non-logarithmic  solution 
of $L_{+}^{x}$ at $x = 0$.
Replacing $t$ with $-t$ gives the corresponding solution $F_{-}^{x}$ of $L_{-}^{x}$.
Then the product
\begin{equation}
	\sqrt{x} \cdot F_{+}^{x} \cdot F_{-}^{x}
	\label{producty}
\end{equation}
is a solution of $\DE_y$.
The fact that $F(y,z)$ also satisfies $\DE_y$ does not directly provide
a functional equality: $F_{\pm}^{x}$ are series at $x=0$, which under \eqref{subs_xt} corresponds to $y = \infty$, but that makes the radius of convergence ($R_y$ from Remark~\ref{Remark1}) equal to 0.
This problem disappears if we take solutions of $L_{\pm}^{x}$ at another singularity, namely $x = -(1+\sqrt{2t^2-1})^4  =  -(1+1/(2\sqrt{-z}))^4$
(note that $\sqrt{-z}$ appears in Theorem~\ref{main}).
The next task is to solve $L_{\pm}^{x}$, which unexpectedly produced a hyperelliptic integral.

\subsection{Solving the differential equation with a Hadamard product}
\label{Had}
To ensure that all ${}_2F_1$ algorithms \cite{vHK19,IvH17}
inside \texttt{hypergeometricsols} from \texttt{Maple}~2021 were applicable,
we reduced the coefficients of $L_{+}^{x}$ to $\mathbb{Q}(x)$ by substituting values for $t$.
However, to our surprise, we found no $\mbox{}_2F_1$ solution.
In an effort to prove this, we computed the monodromy group in Section~\ref{monodromy}.
% These computations could be turned into a more rigorous argument but instead, we
% decided to compute the monodromy which gives the same conclusion (see Corollary~\ref{non-2F1} in Section~\ref{monodromy}).

Since $F(y,z) \in \mathbb{Z}[y]((z))$, we expect equations obtained from $F(y,z)$ to be arithmetic;
we expect that there exists some $c \in \mathbb{Q}(t)\setminus\{0\}$ such that $a_n c^n \in \mathbb{Z}[t]$ for all~$n$.
This turned out to be the case with $c = 2^6 (t^2-1)^4$.
To our surprise\,---\,the intermediate equations $L_{\pm}^z$ and $L_{\pm}^y$ did not have this feature\,---\,not only was $a_n c^n$ in $\mathbb{Z}[t]$,
it was even divisible by $\binom{2n}n$.
Then $\sum \tilde{a}_n  x^n$ with $\tilde{a}_n := a_n / \binom{2n}n$ should also satisfy an arithmetic equation.
One can find this equation with the package \texttt{gfun} in \texttt{Maple}, which can convert $L^x_{+}$ to recurrence for $a_n$, compute a recurrence for $\tilde{a}_n$, and convert back to a differential equation.
By solving that equation we expected to find a $\mbox{}_2F_1$ expression for $\sum \tilde{a}_n x^n$ but, surprisingly, this function is algebraic, similar to~\eqref{tildeu} below.

Let $L^{m} = L_{+}^{x}\big\rvert_{x \mapsto m}$ where $m  :=  \frac1{64x} - (t+1)(t-1)^3$.
Solving $L^m$ is equivalent to solving $L_{+}^{x}$.
The M\"obius transformation $x \mapsto m$ moves the apparent singularity to $x = \infty$ and reduces expression sizes
in equations~\eqref{eqS},~\eqref{eqA} below.
The exponents at $x=0$ are $\frac12, \frac12$ so
$L^m$ has a unique solution of the form $x^{1/2} \sum u_n x^n$ with $u_0 = 1$. A recurrence for $u_n$ can be computed with \texttt{gfun}, it
is given in Theorem~\ref{th1} below.
Now let $\tilde{u}_n := u_n / \binom{2n}n$ and $Z(x) := \sum \tilde{u}_n x^n$.
After computing and solving a differential equation for $Z(x)$ we find
\begin{equation}
	Z(x) = \frac{R_{-}^{1/8} + R_{+}^{1/8}}{2 \sqrt{S}}
	\label{tildeu}
\end{equation}
where
\begin{align}
R_{\pm} \hspace{-2pt} &= A \pm \sqrt{A^2-1}, \nonumber \\
S &= (1 - 16(t + 1)(t - 1)^3x)(1 + 2^6(t^3 + t^2 - t)x - 2^{10}(t^3 + t^2)(t - 1)^3x^2), \label{eqS}
\displaybreak[2]\\
A &=1+2^7(2t-1)^2x-2^{11}(t-1)^3(2t-1)(2t^2+5t-1)x^2 \nonumber
\\ &\quad
+2^{17}t(t-1)^6(2t^2+2t-1)x^3-2^{21} (t^3+t^2)(t-1)^9x^4.  \label{eqA}
\end{align}

For any $X$ the expression $(1+16X)^{1/8}$ is in $\mathbb Z[[X]]$.
Applying this to $R_{\pm}^{1/8} = (1 \pm 16 \sqrt{x} + \cdots)^{1/8}$
we find $R_{\pm}^{1/8}\in\mathbb Z[t][[\sqrt x\,]]$.
Since $R_{-} = R_{+}\rvert_{\sqrt{x} \mapsto -\sqrt{x}}$,
the numerator of~\eqref{tildeu} is in $2 \mathbb{Z}[t][[x]]$
and hence $Z(x) \in \mathbb Z[t][[x]]$.
This establishes an Ap\'ery-like result.

\begin{theorem}
\label{th1}
Define degree $4n$ polynomials $u_n$ by
$u_n=0$ for $n<0$,  $u_0=1$, and
\begin{align}
&
(n + 1)^2u_{n + 1}
\nonumber\\ &\;
- 2^2(16(t^4 - 6t^3 - 4t^2 + 6t - 1)(n^2 + n) + 4t^4 - 24t^3 - 12t^2 + 20t - 3)u_n
\nonumber\\ &\;
- 2^{11}t(t - 1)^3(t + 1)(8(t^2 + 2t - 1)n^2 - 2t^2 - 6t + 3)u_{n - 1}
\nonumber\\ &\;
+ 2^{18}t^2(t - 1)^6(t + 1)^2(2n + 1)(2n - 3)u_{n - 2}
= 0
% \quad\text{for}\;\; n=0,1,2,\dots,
\label{op1}
\end{align}
for $n \geq 0$.
Then $u_n \in \mathbb{Z}[t]$, and is divisible by $\binom{2n}{n}$.
\end{theorem}

Similar results are found by computing recurrences for solutions of $L^m$ or $L^x_{\pm}$ at other singular points.
The smallest recurrence we found this way is \eqref{op1}.

The generating function of $\binom{2n}n$ is $1/\sqrt{1-4x}$, so
we can express $\sum u_n x^n$ (the solution of $L^m$ without the factor $x^{1/2}$)
as
\[
	\sum_{n=0}^{\infty} u_n x^n = \sum_{n=0}^{\infty}  \binom{2n}n \tilde{u}_n \,x^n 
	= \frac1{\sqrt{1-4x}} \star Z(x),
\]
where the Hadamard product $\star$ is given by $\sum a_n x^n \star \sum b_n x^n =  \sum a_n b_n x^n$.

\subsection{Reducing the Hadamard product to a hyperelliptic integral} \label{sec3}
In general, if $u_n = \binom{2n}n \tilde{u}_n$ and $Z(x) = \sum_{n=0}^\infty \tilde{u}_nx^n$ then
\begin{equation}
\sum_{n=0}^\infty u_n x^n =  \frac1{\sqrt{1-4x}} \star Z(x) =\frac1\pi\int_0^{4x} \frac{Z(\xi)}{\sqrt{ \xi (4x-\xi) }} \,\d\xi.
\label{generalstar}
\end{equation}
Now let $Z(x)$ be the algebraic function~\eqref{tildeu}.
The algebraic relation between $Z(x)$ and $x$ is
$((4SZ^2 - 2)^2 - 2)^2 - 2 - 2A = 0$ divided by $1-16(1+t)(1-t)^3 x$.
Remarkably, there is a rational parametrization, however, allowing
a square root makes the parametrization {\em much} shorter:
\begin{align}
x&=\frac{(t^2-v^2)}{16t(tv^4+(t+1)(t-1)^2(t^2-2v^2-t))}, \label{xv} \\
Z&= \frac{(t-1+v)(tv^4+(t+1)(t-1)^2(t^2-2v^2-t)) }{v(v^4-2t^2v^2+(t^2-1)^2)} \sqrt{\frac{v-t}{2t(t^2+tv-1)}}. \label{Zv}
\end{align}
% [x, z^2] = [1/16/t*(t^2-v^2)/(t*v^4+(t+1)*(t-1)^2*(t^2-2*v^2-t)), 1/2/t*(t*v^4+(t+1)*(t-1)^2*(t^2-2*v^2-t))^2*(t-1+v)^2*(v-t)/v^2/(t^2+t*v-1)/(v^4-2*t^2*v^2+(t^2-1)^2)^2]
Substituting the right-hand sides of~\eqref{Zv}, \eqref{xv} for $Z(\xi), \xi$ in~\eqref{generalstar}
% , leaving $x$ as is, gives
gives
\begin{equation}
\frac{\sqrt{-2}}{\pi}\int
	% NOT SURE YET ABOUT THE END POINTS:   _{-t}^{t^{-1}-t} \hspace{-15pt}
	% SIGN WAS ALSO WRONG, SO I REPLACED sqrt{2} BY sqrt{-2}
\frac{(v+t-1) \,\d v}{\sqrt{(v+t) (v+t-t^{-1}) (  v^2-t^2  +  64tx(tv^4+(t+1)(t-1)^2(t^2-2v^2-t))  )}}.
			% THE FACTOR 16 WAS WRONG, SHOULD BE 64
\label{hyperY}
\end{equation}
	% NOT SURE ABOUT THE END POINTS: A linear substitution moves the endpoints $-t$ and $t^{-1}-t$ to $0$ and $1$.
Reinserting the factor $x^{1/2}$ and reversing the transformation $x \mapsto m$
gives a solution of $L^{x}_{+}$. Replacing $t$ by $-t$ gives a solution of $L^{x}_{-}$. Then we get a solution of $\DE_y$ just like~\eqref{producty}.
Reversing~\eqref{subs_xt} to write it in terms of $y$ and $z$ produces a formula similar to~Theorem~\ref{main}.
We made further substitutions for $u$ and $x$ to minimize the expression size, see our website \cite{supp} for details.
To verify that the resulting formula, in Theorem~\ref{main}, is still correct,
we computed a differential equation for it and compared initial conditions.

\section{Curves with real multiplication}
\label{RM-check}
Let $\mathcal{L} = \mathbb{Q}(u, \sqrt{2u^2-2u})$ where $u = 4z$ 
and $\sqrt{2u^2-2u}$ appear in Theorem~\ref{main}.
Parametrization~\eqref{subs_xt} rationalizes this to $\sqrt{2u^2 - 2u} = 2tu$ and $u = 1/(1-2t^2)$,
so $\mathcal{L} = \mathbb{Q}(t)$.
One of the most remarkable aspects is that the hyperelliptic integrals
$I_{\pm}(u,x)$ from Theorem~\ref{main} satisfy second order equations, namely
\begin{equation}
	L_{\pm}^{\text{PF}} := \ L^{x}_{\pm} \big\vert_{x\,\mapsto\,\frac{\,-x\,}{4u^2}} \in   \mathcal{L}(x)[\d / \d x].
	\label{PF}
\end{equation}

The Picard--Fuchs equation for a hyperelliptic integral of genus $g$ almost always has order $2g$, so it is very surprising that $I_{\pm}(u,x)$ satisfies an equation of order~2.
This motivates us to look for special properties of the corresponding hyperelliptic curve~\eqref{Onder}.
We detected multiplication by $\sqrt{2}$ in our monodromy computation in Section~\ref{monodromy}. Here we verify this property with a Humbert relation.

Theorem~\ref{main} contains the square root of
\begin{equation}
H:= v \, (1-v) \, \big( (1-v)(1-u^2v)(1+uv)^2 + x v(1-uv)^2 \big) \in \mathbb{Q}(u,x)[v].
\label{Onder}
\end{equation}
The equation $Y^2 = H(x,u,v)$ defines a family of hyperelliptic curves $C=C^{u,x}$.
Special properties may be visible in the Jacobian $\Jac(C) := H^0(\Omega_C)^*/H_1(C, \mathbb Z)$
which is isomorphic to a two dimensional complex torus $\C^2/\Lambda$. To obtain $\Lambda$, 
choose a symplectic basis $\alpha_1,\beta_1, \alpha_2,\beta_2$ (Figure~\ref{HB} in Section~\ref{monodromy})
for the homology group $H_1(C,\Z)$, and a basis
of holomorphic differentials:
\begin{equation}
	\omega_1 := \frac{\d v}{\sqrt{H}}, \quad \omega_2:=\frac{v\, \d v}{\sqrt{H}}  \in H^0(\Omega_C). \label{omega}
\end{equation}
Now $\Lambda$, the period lattice, is generated over $\Z$ by the columns of the period matrix
\begin{equation} \begin{pmatrix}
\int_{\alpha_1}\omega_1&\int_{\alpha_2}\omega_1&\int_{\beta_1}\omega_1&\int_{\beta_2} \omega_1 \\
\int_{\alpha_1}\omega_2&\int_{\alpha_2}\omega_2&\int_{\beta_1}\omega_2&\int_{\beta_2} \omega_2
\end{pmatrix}. \label{periodmatrix} \end{equation}
After updating the basis $\omega_1,\omega_2$ this becomes
\[ \begin{pmatrix} 1&0&\tau_1&\tau_2\\
            0&1&\tau_2&\tau_3 \end{pmatrix}. \]
Now $\tau_1,\tau_2,\tau_3$ only depend on the chosen symplectic basis.

If the Jacobian has non-trivial endomorphisms, this is reflected
in a \emph{Humbert} relation between $\tau_1$, $\tau_2$ and $\tau_3$, which
is of the form
\[A \tau_1+B\tau_2+C\tau_3+D(\tau_2^2-\tau_1\tau_3)+E=0\]
for some $A,B,C,D,E \in \mathbb{Z}$.
Inside the three-dimensional moduli space $\mathcal{M}_3$ of all curves of
genus $2$ or, what amounts to the same, the space of $\mathcal{A}_2$ of
(principally polarised) abelian surfaces, the curves that satisfy such a
Humbert relation form a \emph{Humbert surface} $H_{\Delta}$ of discriminant
\[\Delta:=B^2-4AC-4DE .\]
For a general curve belonging to $H_{\Delta}$ the endomorphisms of the
Jacobian form an order in the real quadratic field with discriminant $\Delta$,
and we say that the curve (or its Jacobian) has {\em real multiplication} by the order.

By numerical evaluation of the period matrix,
after picking numerical values for $u,x$, one can check that our curves
$C^{u,x}$ belong to the Humbert surface $H_8$:
\[ \Z[\sqrt{2}] \subseteq \operatorname{End}(\Jac(C^{u,x})). \]

However, one can do better than approximate verification. The
insight of Humbert \cite{Hu99} was that, for each $\Delta$, there is a
polynomial condition on the coefficients $a_1,a_2,\allowbreak a_3,a_4$ of the genus~2
curve
\[ Y^2=v(v-a_1)(v-a_2)(v-a_3)(v-a_4)\]
to be on the Humbert surface $H_{\Delta}$. For $\Delta=8$  the condition, already found
by Humbert \cite[p.~327]{Hu99}, reads
\[4a_1a_2a_3a_4((a_1+a_3)(a_2+a_4)-2a_1a_3-2a_2a_4)^2
=(a_2-a_4)^2(a_1-a_3)^2(a_1a_3+a_2a_4)^2.\]
Using this, it is trivial to verify the following statement.

\begin{theorem}
  The smooth curves $C^{u,x}$ admit real multiplication by $\sqrt{2}$\textup. In other words,
  the endomorphism ring of their Jacobian contains $\Z[\sqrt{2}]$.
\end{theorem}

\section{Monodromy calculation} \label{monodromy}

We will now study in some more detail the geometry of the curves $C^{u,x}$.
For each value of $u$, $C^{u,x}$ becomes a 1-parameter family of curves, parametrized by $x$, forming a surface $S^u$. The singular members in this family are
\[x_1=u^2-6u+1+4(1-u)\sqrt{-u},\;\;\;x_2=0,\;\;\;
  x_3=u^2-6u+1-4(1-u)\sqrt{-u},\;\;\; x_4=\infty.\]
These are also the non-apparent singularities of the Picard--Fuchs equation $L^{\text{PF}}$,
the minimal differential equation in $\mathcal{L}(x)[\d / \d x]$ for the periods in~\eqref{periodmatrix}.
It has the expected order $2g = 4$.
Our aim is to find its monodromy representation.

As the local system of solutions to  $L^{\text{PF}}$ can be identified with
the local system of homologies $H_1(C^{u,x},\Z)$ tensored with $\C$, this
explicit form can be obtained by elementary topological means.
To achieve this,  we temporarily fix the value $u=1/2$.
Then the singularities are
\[x_1=-7/4+\sqrt{-2},\;\;\;x_2=0,\;\;\;x_3=-7/4-\sqrt{-2},\;\;\;x_4=\infty.\]
We chose a base point
\[ x_{\text{BP}}=1 \]
because all roots of $H$ are real at $u = 1/2$, $x = 1$. Sorted $v_1<\dots<v_6$ the roots are
\[
v_1 = -3-\sqrt{17}, \;\;
v_2 = \frac{3-\sqrt{17}}2,  \;\;
v_3 = 0,  \;\;
v_4 = 1,  \;\;
v_5 = -3+\sqrt{17},  \;\;
v_6 = \frac{3+\sqrt{17}}2.
\]
To visualize cycles on the hyperelliptic curve $C^{1/2,1}$ we use classical
branch-cuts: in Figure~\ref{HB} we drew oriented cycles $\alpha_1, \beta_1, \alpha_2, \beta_2 \in H_1(C^{1/2,1})$ in 
the $v$-plane. When crossing a branch-cut we change the sheet, and continue the cycle by drawing it as a dashed
line. We chose branch-cuts running from $v_1$ to $v_2$, from $v_3$ to $v_4$ and from $v_5$ to $v_6$.
\begin{figure}[ht]
\begin{center}
\begin{tikzpicture}
\def\sunit{2.4}
\node at (-2*\sunit,0) {$*$}; \node [below] at  (-2*\sunit,0) {$v_1$};
\node at (-1*\sunit,0) {$*$}; \node [below] at  (-1*\sunit,0) {$v_2$};
\node at (0,0) {$*$}; \node [below] at  (0,0) {$v_3$};
\node at (\sunit,0) {$*$}; \node [below] at  (\sunit,0) {$v_4$};
\node at (2*\sunit,0) {$*$}; \node [below] at  (2*\sunit,0) {$v_5$};
\node at (3*\sunit,0) {$*$}; \node [below] at  (3*\sunit,0) {$v_6$};
\draw[line width=1.5pt] (-2*\sunit,0) -- (-1*\sunit,0);
\draw[line width=1.5pt] (0*\sunit,0) -- (1*\sunit,0);
\draw[line width=1.5pt] (2*\sunit,0) -- (3*\sunit,0);
\draw[name path=alpha_1, decoration={ markings,
  mark=at position 0.45 with {\arrow[line width=1.2pt]{>}}},
  postaction={decorate}] (-0.72*\sunit,0) arc (0:360:{0.8*\sunit} and 0.4*\sunit);
\node [above] at (-1.88*\sunit,0.34*\sunit) {$\alpha_1$};
\draw[name path=beta_1, decoration={ markings,
  mark=at position 0.5 with {\arrow[line width=1.2pt]{>}}},
  postaction={decorate}] (0.28*\sunit,0) arc (0:180:{0.8*\sunit} and 0.4*\sunit);
\draw[name path=beta_1d, dashed] (0.28*\sunit,0) arc (0:-180:{0.8*\sunit} and 0.4*\sunit);
\node [above] at (-0.12*\sunit,0.34*\sunit) {$\beta_1$};
\draw[name path=alpha_2, decoration={ markings,
  mark=at position 0.5 with {\arrow[line width=1.2pt]{>}}},
  postaction={decorate}] (2.28*\sunit,0) arc (0:180:{0.8*\sunit} and 0.4*\sunit);
\draw[name path=alpha_2d, dashed] (2.28*\sunit,0) arc (0:-180:{0.8*\sunit} and 0.4*\sunit);
\node [above] at (1.12*\sunit,0.34*\sunit) {$\alpha_2$};
\draw[name path=beta_2, decoration={ markings,
  mark=at position 0.3 with {\arrow[line width=1.2pt]{>}}},
  postaction={decorate}] (3.28*\sunit,0) arc (0:360:{0.8*\sunit} and 0.4*\sunit);
\node [above] at (2.88*\sunit,0.34*\sunit) {$\beta_2$};
\fill [name intersections={of=alpha_1 and beta_1, by={p1}}] (p1) circle (2.1pt);  \node [above] at (p1) {$\vphantom|p_1$};
\fill [name intersections={of=alpha_2 and beta_2, by={p2}}] (p2) circle (2.1pt);  \node [above] at (p2) {$\vphantom|p_2$};
\end{tikzpicture}
\end{center}
\caption{A symplectic homology basis $\alpha_1, \beta_1, \alpha_2, \beta_2 \in H_1(C^{1/2,1}).$}
\label{HB}
\end{figure}
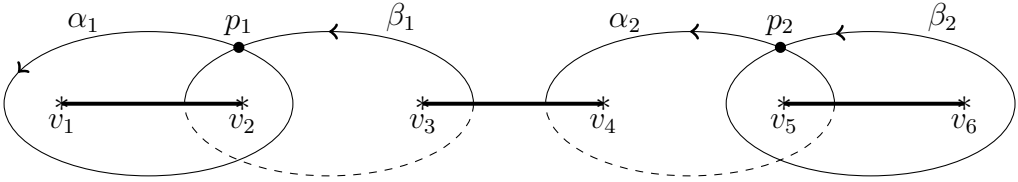

Here the cycle $\alpha_1$ runs counter-clockwise around the cut $v_1 v_2$.
It is completely on the `upper sheet'. The cycle $\beta_1$ encircles $v_2$
and $v_3$ counter-clockwise. This cycle hits the cuts $v_1 v_2$ and $v_3 v_4$,
so when drawing it, the lower part of the cycles is dashed and runs on the
`lower sheet' of the double cover. The cycles $\alpha_1$ and $\beta_1$ intersect
in a single point $p_1$. One has
\[ 1 = \alpha_1 \cdot \beta_1 =-\beta_1 \cdot \alpha_1,\]
as the tangent to $\alpha_1$ at the intersection points $p_1$ has to turn
counter-clockwise towards the tangent of $\beta_1$ at $p_1$.
Similarly, we take cycles $\alpha_2$ and $\beta_2$, where $\alpha_2$
encircles $v_4$ and $v_5$ (so hits two cuts and has lower half-dashed)
and $\beta_2$, which encircles the cut $v_5 v_6$, always counter-clockwise.
Thus we also have
\[ 1 = \alpha_2 \cdot \beta_2= -\beta_2 \cdot \alpha_2 .\]
All other intersection products between the cycles are zero, so that
$\alpha_1, \beta_1, \alpha_2, \beta_2$ make up a symplectic basis for the
homology group $H_1(C^{1/2,1},\Z)$.\\

\begin{figure}[ht]
\begin{center}
\begin{tikzpicture}
\def\sunit{1.8}
\fill (-1*\sunit,1*\sunit) circle (2.1pt); \node [right] at (-1*\sunit,0.9*\sunit) {$x_1$};
\fill (0,0) circle (2.1pt); \node [right] at (0,0) {$x_2$};
\fill (-1*\sunit,-1*\sunit) circle (2.1pt); \node [right] at (-1*\sunit,-0.9*\sunit) {$x_3$};
\fill (1.5*\sunit,0) circle (2.1pt); \node [right] at (1.5*\sunit,0) {$x_{\text{BP}}$};
\draw[decoration={markings, mark=at position 0.3 with {\arrow[line width=1.2pt]{>}}}, postaction={decorate}]
 (1.5*\sunit,0) .. controls (-1.0*\sunit,1.35*\sunit) .. (-1.2*\sunit,1.1*\sunit) .. controls (-1.25*\sunit,0.8*\sunit) .. (1.5*\sunit,0);
\draw[decoration={markings, mark=at position 0.3 with {\arrow[line width=1.2pt]{>}}}, postaction={decorate}]
 (1.5*\sunit,0) .. controls (-0.1*\sunit,0.2*\sunit) .. (-0.15*\sunit,0) .. controls (-0.1*\sunit,-0.2*\sunit) .. (1.5*\sunit,0);
\draw[decoration={markings, mark=at position 0.3 with {\arrow[line width=1.2pt]{>}}}, postaction={decorate}]
 (1.5*\sunit,0) .. controls (-1.25*\sunit,-0.8*\sunit) .. (-1.2*\sunit,-1.1*\sunit) .. controls (-1.0*\sunit,-1.35*\sunit) .. (1.5*\sunit,0);
\draw[decoration={markings, mark=at position 0.5 with {\arrow[line width=1.2pt]{>}}}, postaction={decorate}]
 (1.5*\sunit,0) arc (360:0:{2*\sunit} and 1.35*\sunit);
\node [above] at (0.55*\sunit,0.5*\sunit) {$\ell_1$};
\node [left] at (-0.1*\sunit,-0.1*\sunit) {$\ell_2$};
\node [below] at (0.55*\sunit,-0.5*\sunit) {$\ell_3$};
\node [right] at (-2.3*\sunit,0.5*\sunit) {$\ell_4$};
\end{tikzpicture}
\end{center}
\caption{Loops around the singularities for $u=1/2$}
\label{Pos0}
\end{figure}
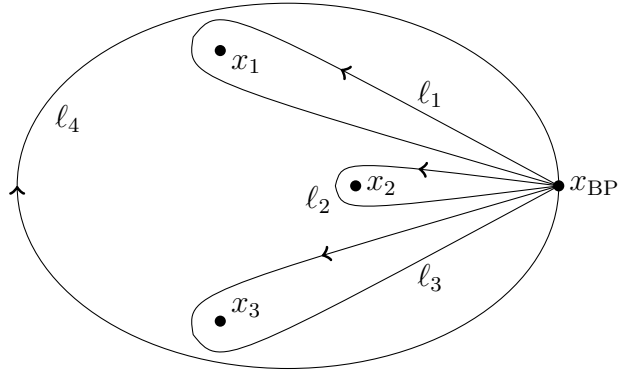

Next we examine how continuously changing $x$ (while fixing $u = 1/2$) impacts Figure~\ref{HB}.
We choose loops $\ell_1, \ell_2, \ell_3, \ell_4$
in the $x$-plane that start at $x_{\text{BP}}=1$ and loop around $x_1=-7/4+\sqrt{-2}$, $x_2=0$, $x_3=-7/4-\sqrt{-2}$, $x_4 = \infty$ respectively, see Figure~\ref{Pos0}.
Following such a loop, the points $v_1,\ldots,v_6$ in Figure~\ref{HB} move, dragging the paths $\alpha_1, \beta_1, \alpha_2, \beta_2$ along with them. When we return to $x_{\text{BP}}$
at the end of the loop, we have a permutation of the original roots $v_1,\ldots,v_6$, and the dragged paths form a new symplectic basis.
The change of basis matrix, between the new and the original basis, is a
symplectic matrix $M_i \in \operatorname{Sp}_4(\Z) \subset \operatorname{GL}_4(\Z)$ for each $\ell_i$.
This gives a homomorphism (the  {\em monodromy representation})
\[ \rho\colon \pi_1(\mathbb P^1\setminus\{x_1,x_2,x_3,x_4\}, x_{\text{BP}}) \to  \operatorname{Sp}_4(\Z)  \cong \operatorname{Aut}(H_1(C^{1/2,1},\Z)), \]
with $\rho( \ell_i )  = M_i$.
As the composed path $\ell_1 \ell_2 \ell_3 \ell_4$ is contractible to a point we have 
\begin{equation}
M_1 M_2 M_3 M_4=\operatorname{Id}.
\label{product1}
\end{equation}

\begin{prop}\label{monodromymatrices}
With the basis for $H_1(C^{1/2,1},\Z)$ in Figure~\textup{\ref{HB}}
and generators $\ell_1,\ldots,\ell_4$ for $\pi_1(\C\setminus\{x_1,x_2,x_3,x_4\},x_{\textup{BP}})$ in Figure~\textup{\ref{Pos0}} we have%
\footnote{This monodromy is close to the type that is classified in \cite[Theorem~3.1]{BR13}, with a subtle difference. The monodromy in \cite{BR13} is irreducible
over $\mathbb{C}$, whereas the monodromy here is irreducible over~$\mathbb{Q}$ but reducible over $\mathbb{C}$, see Section~\ref{mono-RM}.  This subtle difference
is important because reducibility over $\mathbb{C}$ is what makes it possible for a hyperelliptic integral to satisfy a differential equation of order~2 instead of the expected order~4.
It also causes the real multiplication discussed in Section~\ref{RM-check}.}
\begin{alignat*}{2}
M_1 &:= \begin{pmatrix}		% Change from version 7o:  Swapped M1 <--> M3
      0&-1&1&1\\
      2&2&-1&-2\\
      2&1&0&-2\\
      -1 & -1 & 1 & 2
    \end{pmatrix}, &\quad
M_2 &:= \begin{pmatrix}
      1&-1&0&0\\
      0&1&0&0\\
      0&0&1&-2\\
      0&0&0&1
    \end{pmatrix},
\displaybreak[2]\\
M_3 &:=\begin{pmatrix}
      2&-1&1&-1\\
      2&0&1&-2\\
      2&-1&2&-2\\
      1&-1&1&0
     \end{pmatrix}, &\quad
M_4 &:= \begin{pmatrix}
      -1&0&0&0\\
      -2&-1&0&0\\
      0&0&-1&0\\
      0&0&-1&-1
    \end{pmatrix}.
\end{alignat*}  
\end{prop}

\begin{proof}
To compute these monodromy matrices we
use Picard--Lefsch\-etz theory \cite{HZ77}, \cite{Va02}, which tells us that $M_i$
is determined by vanishing cycles. In general, when a curve
acquires an $A_1$-singularity at a certain value of the parameter, there is a
{\em vanishing cycle $\delta$}, that contracts to a point.
The monodromy is geometrically a Dehn twist along this cycle. On the
level of homology, the monodromy transformation $T$ is a
{\em symplectic reflection} in~$\delta$:
\begin{equation}
	T_{\delta}\colon \gamma \mapsto \gamma -(\delta \cdot \gamma) \delta.   \label{Tdelta}
\end{equation}

A \texttt{Maple}-animation showed that if $x$ follows loop $\ell_1$ around $x_1$
then the roots $v_1,\ldots,v_6$ of $H$ move as indicated in Figure~\ref{Animation1}.
There are {\em two} $A_1$-singularities at $x = x_1$, and two disjoint {\em vanishing cycles} $\delta_1, \delta_2$.
In such cases, the transformations $T_{\delta_1}$ and $T_{\delta_2}$ commute.  We will compute them next; the
monodromy will be their product.

%%%%%%%%%%%%%%%%%%%%  M1 (same as M3 from 7o.tex) %%%%%%%%%%%%%%%%%%%%

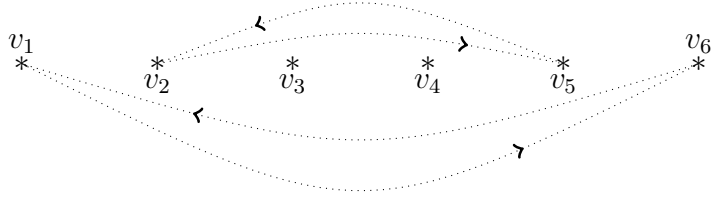
\begin{figure}[ht]
\begin{center}
\begin{tikzpicture}
\def\sunit{1.8}
\node at (-2*\sunit,0) {$*$}; \node [above] at  (-2*\sunit,0) {$v_1$};
\node at (-1*\sunit,0) {$*$}; \node [below] at  (-1*\sunit,0) {$v_2$};
\node at (0,0) {$*$}; \node [below] at  (0,0) {$v_3$};
\node at (\sunit,0) {$*$}; \node [below] at  (\sunit,0) {$v_4$};
\node at (2*\sunit,0) {$*$}; \node [below] at  (2*\sunit,0) {$v_5$};
\node at (3*\sunit,0) {$*$}; \node [above] at  (3*\sunit,0) {$v_6$};
\node at (-3*\sunit,0) {\phantom.};
\node at (4*\sunit,0) {\phantom.};
\draw[dotted, decoration={ markings,
  mark=at position 0.375 with {\arrow[line width=1.2pt]{>}},
  mark=at position 0.875 with {\arrow[line width=1.2pt]{>}}},
  postaction={decorate}]
 (-1*\sunit,0) .. controls (0.5*\sunit,0.3*\sunit) .. (2*\sunit,0) .. controls (0.5*\sunit,0.6*\sunit) .. (-1*\sunit,0);
\draw[dotted, decoration={ markings,
  mark=at position 0.375 with {\arrow[line width=1.2pt]{>}},
  mark=at position 0.875 with {\arrow[line width=1.2pt]{>}}},
  postaction={decorate}]
 (-2*\sunit,0) .. controls (0.5*\sunit,-1.25*\sunit) .. (3*\sunit,0) .. controls (0.5*\sunit,-0.75*\sunit) .. (-2*\sunit,0);
\end{tikzpicture}
\end{center}
\caption{Motion of $v_1,\ldots,v_6$ when $x$ follows $\ell_{{1}}$}
\label{Animation1}
\end{figure}

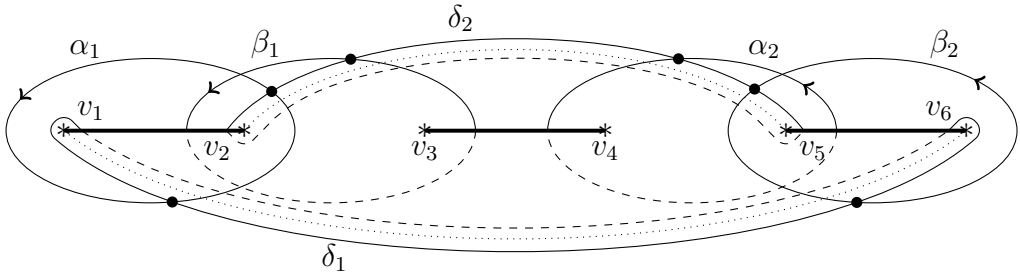
\begin{figure}[ht]
\begin{center}
% for animation for $\ell_1$
\begin{tikzpicture}
\def\sunit{2.4}
\node at (-2*\sunit,0) {$*$}; \node [above] at  (-1.85*\sunit,0) {$v_1$};
\node at (-1*\sunit,0) {$*$}; \node [below] at  (-1.15*\sunit,-0) {$v_2$};
\node at (0,0) {$*$}; \node [below] at  (0,0) {$v_3$};
\node at (\sunit,0) {$*$}; \node [below] at  (\sunit,0) {$v_4$};
\node at (2*\sunit,0) {$*$}; \node [below] at  (2.15*\sunit,0) {$v_5$};
\node at (3*\sunit,0) {$*$}; \node [above] at  (2.85*\sunit,0) {$v_6$};
\draw[line width=1.5pt] (-2*\sunit,0) -- (-1*\sunit,0);
\draw[line width=1.5pt] (0*\sunit,0) -- (1*\sunit,0);
\draw[line width=1.5pt] (2*\sunit,0) -- (3*\sunit,0);
\draw[name path=alpha_1, decoration={ markings,
  mark=at position 0.45 with {\arrow[line width=1.2pt]{>}}},
  postaction={decorate}] (-0.72*\sunit,0) arc (0:360:{0.8*\sunit} and 0.4*\sunit);
\node [above] at (-1.88*\sunit,0.34*\sunit) {$\alpha_1$};
\draw[name path=beta_1, decoration={ markings,
  mark=at position 0.87 with {\arrow[line width=1.2pt]{>}}},
  postaction={decorate}] (0.28*\sunit,0) arc (0:180:{0.8*\sunit} and 0.4*\sunit);
\draw[name path=beta_1d, dashed] (0.28*\sunit,0) arc (0:-180:{0.8*\sunit} and 0.4*\sunit);
\node [above] at (-0.88*\sunit,0.34*\sunit) {$\beta_1$};
\draw[name path=alpha_2, decoration={ markings,
  mark=at position 0.17 with {\arrow[line width=1.2pt]{>}}},
  postaction={decorate}] (2.28*\sunit,0) arc (0:180:{0.8*\sunit} and 0.4*\sunit);
\draw[name path=alpha_2d, dashed] (2.28*\sunit,0) arc (0:-180:{0.8*\sunit} and 0.4*\sunit);
\node [above] at (1.88*\sunit,0.34*\sunit) {$\alpha_2$};
\draw[name path=beta_2, decoration={ markings,
  mark=at position 0.1 with {\arrow[line width=1.2pt]{>}}},
  postaction={decorate}] (3.28*\sunit,0) arc (0:360:{0.8*\sunit} and 0.4*\sunit);
\node [above] at (2.88*\sunit,0.34*\sunit) {$\beta_2$};
%\fill [name intersections={of=alpha_1 and beta_1, by={p1}}] (p1) circle (2.1pt);
%\fill [name intersections={of=alpha_2 and beta_2, by={p2}}] (p2) circle (2.1pt);
%
\draw[dotted] (-2*\sunit,0) .. controls (-1*\sunit,-0.8*\sunit) and (2*\sunit,-0.8*\sunit) .. (3*\sunit,0);
\draw[name path=delta_1, rounded corners=4.5pt] (-1.9*\sunit,0) -- (-2*\sunit,0.1*\sunit) -- (-2.1*\sunit,0) .. controls (-1.1*\sunit,-0.88*\sunit) and (2.1*\sunit,-0.88*\sunit) .. (3.1*\sunit,0) -- (3*\sunit,0.1*\sunit) -- (2.9*\sunit,0);
\draw[dashed] (-1.9*\sunit,0) .. controls (-0.9*\sunit,-0.72*\sunit) and (1.9*\sunit,-0.72*\sunit) .. (2.9*\sunit,0);
\fill [name intersections={of=alpha_1 and delta_1, by={q1}}] (q1) circle (2.1pt);
\fill [name intersections={of=beta_2 and delta_1, by={q2}}] (q2) circle (2.1pt);
\node [below] at (-0.5*\sunit,-0.57*\sunit) {$\delta_{{1}}$};
\draw[dotted] (-1*\sunit,0) .. controls (-0.5*\sunit,0.6*\sunit) and (1.5*\sunit,0.6*\sunit) .. (2*\sunit,0);
\draw[name path=delta_2] (-1.1*\sunit,0) .. controls (-0.6*\sunit,0.68*\sunit) and (1.6*\sunit,0.68*\sunit) .. (2.1*\sunit,0);
\draw[dashed, rounded corners=4.5pt] (-1.1*\sunit,0) -- (-1*\sunit,-0.1*\sunit) -- (-0.9*\sunit,0) .. controls (-0.4*\sunit,0.52*\sunit) and (1.4*\sunit,0.52*\sunit) .. (1.9*\sunit,0) -- (2*\sunit,-0.1*\sunit) -- (2.1*\sunit,0);
\fill [name intersections={of=alpha_1 and delta_2, by={q3}}] (q3) circle (2.1pt);
\fill [name intersections={of=beta_1 and delta_2, by={q4}}] (q4) circle (2.1pt);
\fill [name intersections={of=alpha_2 and delta_2, by={q5}}] (q5) circle (2.1pt);
\fill [name intersections={of=beta_2 and delta_2, by={q6}}] (q6) circle (2.1pt);
\node [above] at (0.2*\sunit,0.49*\sunit) {$\delta_{{2}}$};
\end{tikzpicture}
\end{center}
\caption{Vanishing cycles when $x$ follows $\ell_1$}  % REPLACED 3 with 1
\label{VS-3}
\end{figure}

Figure~\ref{Animation1} shows $v_1$ and $v_6$ interchanging in the lower part of the
plane, in a {\em counter clockwise}\footnote{All branch-point swaps in this paper happen to be counter clockwise.
(For clockwise swaps, replace the minus sign in~(\ref{Tdelta}) by a plus sign, this inverts $T_{\delta}$.)
Unlike swaps, a vanishing cycle $\delta$ does not have a natural orientation. Indeed, $T_{\delta} = T_{-\delta}$ so
a clockwise treatment of $\delta$ produces the same monodromy as our counter clockwise treatment.}
manner (i.e. the point that travels left moves {\em over} the point that travels right).
Drawing an auxiliary arc in the lower part of the plane, connecting $v_1$ and $v_6$,
gives the vanishing cycle $\delta_{{1}}$ shown in Figure~\ref{VS-3}.
It encircles $v_1$ and $v_6$ and is half-dashed, as it intersects the cuts $v_1 v_2$ and $v_5 v_6$.
Figure~\ref{Animation1} also shows $v_2$ and $v_5$ interchanging in the upper part of the plane.
That corresponds to the vanishing cycle $\delta_{{2}}$ in Figure~\ref{VS-3}.

Careful comparison shows that $\delta_{{1}} =-(\beta_1 + \alpha_2)$ in the homology.
To ensure correctness, we prove this formula by checking
that $\delta_1 \cdot \gamma = -(\beta_1+\alpha_2) \cdot \gamma$
for each $\gamma$ in our basis $\alpha_1,\beta_1,\alpha_2,\beta_2$.
Next we apply
\[ T_{\delta_1}\colon \gamma \mapsto \gamma -(\delta_1 \cdot \gamma) \delta_1\]
to each element in our basis, resulting in the matrix
\[
  T_{\delta_1}=\begin{pmatrix}
      1&0&0&0\\
      1&1&0&-1\\
      1&0&1&-1\\
      0&0&0&1
   \end{pmatrix}.
\]
Similarly we find and prove $\delta_2  = -\alpha_1 +\beta_1 + \alpha_2-\beta_2$ with the corresponding symplectic reflection
\[
  T_{\delta_{{2}}}=\begin{pmatrix}
      0&-1&1&1\\
      1&2&-1&-1\\
      1&1&0&-1\\
      -1&-1&1&2
   \end{pmatrix}.
\]
We verify that $T_{\delta_{{1}}}$ and $T_{\delta_{{2}}}$ commute, as they should, and obtain the monodromy
by taking their product
\[
 M_{{1}} := T_{\delta_{{1}}} T_{\delta_{{2}}} = \begin{pmatrix}
      0&-1&1&1\\
      2&2&-1&-2\\
      2&1&0&-2\\
      -1&-1&1&2
   \end{pmatrix}.
\]

%%%%%%%%%%%%%%%%%%%%  M2 (same as M2 from 7o.tex) %%%%%%%%%%%%%%%%%%%%

The first step towards finding $M_2$ is Figure~\ref{A2}. This time $v_1$ and $v_2$ change
place, with $v_2$ going {\em over} $v_1$. At the same time $v_5$ makes a complete counterclockwise turn around $v_4$.%
\footnote{The degeneration of $C^{u,x}$ when $x$ approaches $x_2$
is slightly different from the other cases.
The local monodromy around $x_2$ is the product of a
symplectic reflection and the \emph{square} of another
(commuting with the first) symplectic reflection. This
leads then to a coefficient~$2$ in the Jordan form.
For $x=x_2=0$ the two double factors appearing in the equation
$(v-1)^2$ and $(uv+1)^2$ behave differently: the point
$(x=0,y=0,v=1)$ is a \emph{singular} point of the surface~$S^u$,
whereas the point $(x=0,y=0,v=-1/u)$ is a non-singular point.
In local coordinates putting the singularity at the
origin, the first point has normal form
$y^2=v^2+x^2$ (singularity of the surface),
whereas the second  $y^2=v^2+x$ (smooth point on the surface).
Because we have $x^2$ in the first normal form, the
local monodromy is the \emph{square} of a symplectic
reflection.}
The auxillary arc connecting $v_1$ and $v_2$ can be taken to be the
branch-cut $v_1 v_2$, so the vanishing cycle is simply $\alpha_1$.

\begin{figure}[ht]
\begin{center}
\begin{tikzpicture}
\def\sunit{1.8}
\node at (-2*\sunit,0) {$*$}; \node [above] at  (-2*\sunit,0) {$v_1$};
\node at (-1*\sunit,0) {$*$}; \node [right] at  (-1*\sunit,0) {$v_2$};
\node at (0,0) {$*$}; \node [below] at  (0,0) {$v_3$};
\node at (\sunit,0) {$*$}; \node [right] at  (\sunit,0) {$v_4$};
\node at (2*\sunit,0) {$*$}; \node [below] at  (2*\sunit,0) {$v_5$};
\node at (3*\sunit,0) {$*$}; \node [left] at  (3*\sunit,0) {$v_6$};
\draw[dotted, decoration={ markings,
  mark=at position 0.375 with {\arrow[line width=1.2pt]{>}},
  mark=at position 0.875 with {\arrow[line width=1.2pt]{>}}},
  postaction={decorate}]
 (-2*\sunit,0) .. controls (-1.5*\sunit,-0.3*\sunit) .. (-1*\sunit,0) .. controls (-1.5*\sunit,0.3*\sunit) .. (-2*\sunit,0);
\draw[dotted, decoration={ markings,
  mark=at position 0.75 with {\arrow[line width=1.2pt]{>}}},
  postaction={decorate}]
 (2*\sunit,0) .. controls (.1*\sunit,0.9*\sunit) and (.1*\sunit,-0.9*\sunit) .. (2*\sunit,0);
\draw[dashed, decoration={ markings,
  mark=at position 0.75 with {\arrow[line width=1.2pt]{>}}},
  postaction={decorate}]
 (3*\sunit,0) .. controls (4.2*\sunit,-0.75*\sunit) and (4.2*\sunit,0.75*\sunit) .. (3*\sunit,0);
\end{tikzpicture}
\end{center}
\vskip-\baselineskip
\caption{Motion of $v_1,\ldots,v_6$ when $x$ follows $\ell_2$}
\label{A2}
\end{figure}
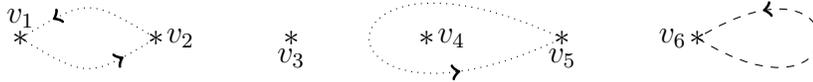

The matrix of the map $T_{\alpha_1}\colon \gamma \mapsto \gamma-(\alpha_1\cdot \gamma) \alpha_1$
is
\[T_{\alpha_1}=\begin{pmatrix}
      1&-1&0&0\\
      0&1&0&0\\
      0&0&1&0\\
      0&0&0&1
   \end{pmatrix}.
\]
The loop of $v_5$ around $v_4$ can be seen as the {\em square} of
the interchanging $v_4, v_5$\,---\,this is just the square of $T_{\alpha_2}$!
\[ T_{\alpha_2} = \begin{pmatrix}
      1&0&0&0\\
      0&1&0&0\\
      0&0&1&-1\\
      0&0&0&1
   \end{pmatrix}.
\]
Thus the monodromy $M_2$ corresponding to the loop $\ell_2$ is
\[
 M_2 :=T_{\alpha_1} T^2_{\alpha_2} = \begin{pmatrix}
      1&-1&0&0\\
      0&1&0&0\\
      0&0&1&-2\\
      0&0&0&1
   \end{pmatrix}.
\]

%%%%%%%%%%%%%%%%%%%%  M3 (same as M1 from 7o.tex) %%%%%%%%%%%%%%%%%%%%

Loop $\ell_3$ is very similar to loop $\ell_1$, the motion of $v_1,\ldots,v_6$ differs in a subtle way, as shown in Figure~\ref{Animation3}.
There are again two $A_1$-singularities, and two disjoint vanishing cycles $\delta_{{3}}, \delta_{{4}}$.

\begin{figure}[ht]
\begin{center}
\begin{tikzpicture}
\def\sunit{1.8}
\node at (-2*\sunit,0) {$*$}; \node [below] at  (-2*\sunit,0) {$v_1$};
\node at (-1*\sunit,0) {$*$}; \node [above] at  (-1*\sunit,0) {$v_2$};
\node at (0,0) {$*$}; \node [above] at  (0,0) {$v_3$};
\node at (\sunit,0) {$*$}; \node [above] at  (\sunit,0) {$v_4$};
\node at (2*\sunit,0) {$*$}; \node [above] at  (2*\sunit,0) {$v_5$};
\node at (3*\sunit,0) {$*$}; \node [below] at  (3*\sunit,0) {$v_6$};
\draw[dotted, decoration={ markings,
  mark=at position 0.375 with {\arrow[line width=1.2pt]{>}},
  mark=at position 0.875 with {\arrow[line width=1.2pt]{>}}},
  postaction={decorate}]
 (-1*\sunit,0) .. controls (0.5*\sunit,-0.6*\sunit) .. (2*\sunit,0) .. controls (0.5*\sunit,-0.3*\sunit) .. (-1*\sunit,0);
\draw[dotted, decoration={ markings,
  mark=at position 0.375 with {\arrow[line width=1.2pt]{>}},
  mark=at position 0.875 with {\arrow[line width=1.2pt]{>}}},
  postaction={decorate}]
 (-2*\sunit,0) .. controls (0.5*\sunit,0.75*\sunit) .. (3*\sunit,0) .. controls (0.5*\sunit,1.25*\sunit) .. (-2*\sunit,0);
\end{tikzpicture}
\end{center}
\caption{Motion of $v_1,\ldots,v_6$ when $x$ follows $\ell_{{3}}$}
\label{Animation3}
\end{figure}
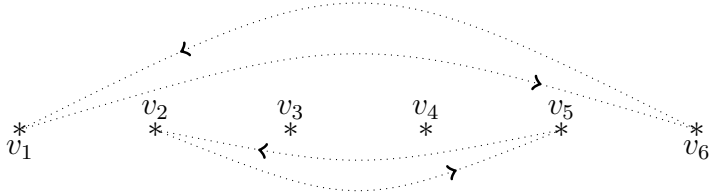

Figure~\ref{Animation3} shows $v_1$ and $v_6$ changing place. They move in the upper part of the plane and $v_6$ moves {\em over} $v_1$.
At the same time, $v_2$ and $v_5$ change place as well (again, $v_5$ moves {\em over}~$v_2$).
If $x$ approaches the $x_3$, then $v_1$ and $v_6$ coalesce. Now draw an auxiliary arc in the upper part of the plane, connecting
$v_1$ and $v_6$. The vanishing cycle $\delta_{{4}}$ is then the cycle that
encircles this arc (counter-clockwise) at small distance (see Figure~\ref{VS-1}).
%	MENTION:    ^^^^^^^^^^^^^^^^^  DOES NOT MATTER

\begin{figure}[ht]
\begin{center}
% for animation for $\ell_3$
\begin{tikzpicture}
\def\sunit{2.4}
\node at (-2*\sunit,0) {$*$}; \node [below] at  (-1.85*\sunit,0) {$v_1$};
\node at (-1*\sunit,0) {$*$}; \node [below] at  (-1.15*\sunit,0) {$v_2$};
\node at (0,0) {$*$}; \node [below] at  (0,0) {$v_3$};
\node at (\sunit,0) {$*$}; \node [below] at  (\sunit,0) {$v_4$};
\node at (2*\sunit,0) {$*$}; \node [below] at  (2.15*\sunit,0) {$v_5$};
\node at (3*\sunit,0) {$*$}; \node [below] at  (2.85*\sunit,0) {$v_6$};
\draw[line width=1.5pt] (-2*\sunit,0) -- (-1*\sunit,0);
\draw[line width=1.5pt] (0*\sunit,0) -- (1*\sunit,0);
\draw[line width=1.5pt] (2*\sunit,0) -- (3*\sunit,0);
\draw[name path=alpha_1, decoration={ markings,
  mark=at position 0.45 with {\arrow[line width=1.2pt]{>}}},
  postaction={decorate}] (-0.72*\sunit,0) arc (0:360:{0.8*\sunit} and 0.4*\sunit);
\node [above] at (-1.88*\sunit,0.34*\sunit) {$\alpha_1$};
\draw[name path=beta_1, decoration={ markings,
  mark=at position 0.25 with {\arrow[line width=1.2pt]{>}}},
  postaction={decorate}] (0.28*\sunit,0) arc (0:180:{0.8*\sunit} and 0.4*\sunit);
\draw[name path=beta_1d, dashed] (0.28*\sunit,0) arc (0:-180:{0.8*\sunit} and 0.4*\sunit);
\node [above] at (0.2*\sunit,0.24*\sunit) {$\beta_1$};
\draw[name path=alpha_2, decoration={ markings,
  mark=at position 0.72 with {\arrow[line width=1.2pt]{>}}},
  postaction={decorate}] (2.28*\sunit,0) arc (0:180:{0.8*\sunit} and 0.4*\sunit);
\draw[name path=alpha_2d, dashed] (2.28*\sunit,0) arc (0:-180:{0.8*\sunit} and 0.4*\sunit);
\node [above] at (0.8*\sunit,0.24*\sunit) {$\alpha_2$};
\draw[name path=beta_2, decoration={ markings,
  mark=at position 0.1 with {\arrow[line width=1.2pt]{>}}},
  postaction={decorate}] (3.28*\sunit,0) arc (0:360:{0.8*\sunit} and 0.4*\sunit);
\node [above] at (2.88*\sunit,0.34*\sunit) {$\beta_2$};
%\fill [name intersections={of=alpha_1 and beta_1, by={p1}}] (p1) circle (2.1pt);
%\fill [name intersections={of=alpha_2 and beta_2, by={p2}}] (p2) circle (2.1pt);
%
\draw[dotted] (-2*\sunit,0) .. controls (-1*\sunit,0.8*\sunit) and (2*\sunit,0.8*\sunit) .. (3*\sunit,0);
\draw[name path=delta_1, rounded corners=4.5pt] (-1.9*\sunit,0) -- (-2*\sunit,-0.1*\sunit) -- (-2.1*\sunit,0) .. controls (-1.1*\sunit,0.88*\sunit) and (2.1*\sunit,0.88*\sunit) .. (3.1*\sunit,0) -- (3*\sunit,-0.1*\sunit) -- (2.9*\sunit,0);
\draw[dashed] (-1.9*\sunit,0) .. controls (-0.9*\sunit,0.72*\sunit) and (1.9*\sunit,0.72*\sunit) .. (2.9*\sunit,0);
\fill [name intersections={of=alpha_1 and delta_1, by={q1}}] (q1) circle (2.1pt);
\fill [name intersections={of=beta_2 and delta_1, by={q2}}] (q2) circle (2.1pt);
\node [above] at (-0.5*\sunit,0.57*\sunit) {$\delta_{{4}}$};
\draw[dotted] (-1*\sunit,0) .. controls (-0.5*\sunit,-0.6*\sunit) and (1.5*\sunit,-0.6*\sunit) .. (2*\sunit,0);
\draw[dashed, name path=delta_3d] (-1.1*\sunit,0) .. controls (-0.6*\sunit,-0.68*\sunit) and (1.6*\sunit,-0.68*\sunit) .. (2.1*\sunit,0);
\draw[name path=delta_3, rounded corners=4.5pt] (-1.1*\sunit,0) -- (-1*\sunit,0.1*\sunit) -- (-0.9*\sunit,0) .. controls (-0.4*\sunit,-0.52*\sunit) and (1.4*\sunit,-0.52*\sunit) .. (1.9*\sunit,0) -- (2*\sunit,0.1*\sunit) -- (2.1*\sunit,0);
\fill [name intersections={of=alpha_1 and delta_3, by={q3}}] (q3) circle (2.1pt);
\fill [name intersections={of=beta_1d and delta_3d, by={q4}}] (q4) circle (2.1pt);
\fill [name intersections={of=alpha_2d and delta_3d, by={q5}}] (q5) circle (2.1pt);
\fill [name intersections={of=beta_2 and delta_3, by={q6}}] (q6) circle (2.1pt);
\node [below] at (0.2*\sunit,-0.49*\sunit) {$\delta_{{3}}$};
\end{tikzpicture}
\end{center}
\caption{Vanishing cycles when $x$ follows $\ell_{{3}}$}
\label{VS-1}
\end{figure}
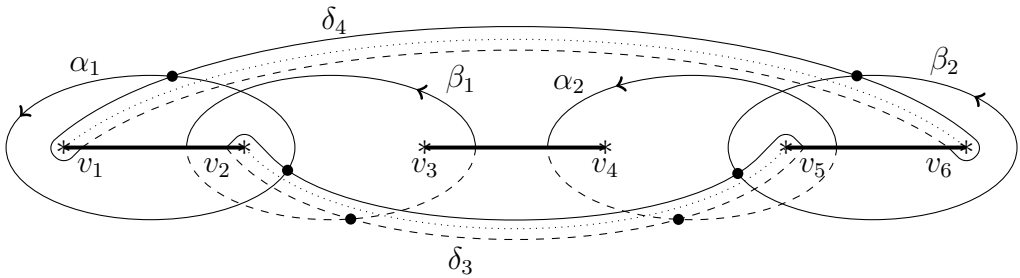

The vanishing cycle $\delta_{{3}}$ comes from the coalescence of $v_2$ and $v_5$ at the singularity. We take an auxiliary arc, connecting
$v_2$ and $v_5$, now in the lower part of the plane, and $\delta_{{3}}$ runs at
close distance counterclockwise around this arc. Like before, we find and prove $\delta_{{3}} =\alpha_1 +\beta_1 +\alpha_2+\beta_2$.
We apply $T_{\delta_{{3}}}\colon \gamma \mapsto \gamma -(\delta_{{3}}\cdot \gamma) \delta_{{3}}$ to each element of our basis, resulting in the matrix.
\[
  T_{\delta_{{3}}}=\begin{pmatrix}
      2&-1&1&-1\\
      1&0&1&-1\\
      1&-1&2&-1\\
      1&-1&1&0
   \end{pmatrix}.
\]
Since $\delta_4 = \beta_1 + \alpha_2$ equals $-\delta_1$  from $\ell_1$ we get $T_{\delta_4} = T_{\delta_1}$.
We verified that $T_{\delta_{{3}}}$ and $T_{\delta_{{4}}}$ commute, as they should.
Their product is:
\[M_{{3}} := T_{\delta_{{3}}} T_{\delta_{{4}}} = \begin{pmatrix}
      2&-1&1&-1\\
      2&0&1&-2\\
      2&-1&2&-2\\
      1&-1&1&0
   \end{pmatrix}.
\]

%%%%%%%%%%%%%%%%%  M4 (same as in previous versions) %%%%%%%%%%%%%%%%%

Finally, $M_4 := (M_1 M_2 M_3)^{-1}$
has the expected Jordan-structure consisting of two $2\times 2$ blocks,
like matrices $M_1, M_2, M_3$, except that the eigenvalues are $-1$ instead of~1.  To guard against mistakes, we independently computed each matrix with numerical analytic continuation,
and found the same $M_1,\ldots,M_4$ to high precision~\cite{supp}.
\end{proof}

\iffalse

	The braid section was here, but I'm moving it further
	below, after section "Monodromy and real multiplication"
	so that the braid text has access to the 2x2 matrices.

\fi

\subsection{Monodromy and real multiplication} \label{mono-RM}
With $\mathcal{L}$ as in Section~\ref{RM-check}, the integrals from Theorem~\ref{main} are
\begin{equation*}
	I_{\pm}(u,x) =  \frac1{\pi} \int_0^1 \omega_1 +  c_{\pm}  \omega_2,
\end{equation*}
where $c_{\pm} = -u  \pm \sqrt{2u^2-2u} \in \mathcal{L}$
are linear combinations of the periods, and hence they are solutions of $L^{\text{PF}}$ (from Section~\ref{monodromy}).
Indeed,
\begin{equation}
	L^{\text{PF}} = \operatorname{LCLM}(L_+^{\text{PF}}, L_-^{\text{PF}}) \in \mathcal{L}(x)[\d / \d x].
	\label{PFdec}
\end{equation}

The remarkable feature that our hyperelliptic integrals $I_{\pm}(u,x)$ manage to satisfy 2nd order equations, instead of the expected $2g = 4$,
corresponds to the fact that 
the $4$th order Picard--Fuchs operator $L^{\text{PF}}$ decomposes into two second-order
operators (equation~\eqref{PFdec}). 
This in turn is equivalent to $\mathbb{C}^4$ decomposing as a direct sum of 2-dimensional $G$-invariant subspaces,
where $G := \langle M_1,M_2,M_3\rangle$ is the {\em monodromy group} of $L^{\text{PF}}$.
It is not \emph{a priori} obvious that this property is possible for a family of hyperelliptic curves.

The two $G$-invariant subspaces are defined over $\mathbb{Q}(\sqrt{2})$. That suggests that $C = C^{u,x}$ has real multiplication by $\sqrt{2}$,
i.e., $\mathbb Z[\sqrt{2}] \subseteq \operatorname{End}(\Jac(C))$, as was verified in Section~\ref{RM-check}.
Endomorphisms of $\Jac(C)$ induce endomorphisms of the lattice $H_1(C, \mathbb Z)=\mathbb Z^4$
that commute with the monodromy representation
\[
\rho\colon\pi_1(\mathbb C \setminus \{x_1,x_2,x_3,x_4\},x_{\text{BP}}) \to \operatorname{Sp}_4(\mathbb Z),
\]
where $\rho(\ell_i) = M_i$.

For a representation $\rho\colon G \to \operatorname{GL}_n(\mathbb Z)$,
the endomorphism ring $\operatorname{End}(\rho)$ consists of all matrices $A \in \operatorname{GL}_n(\mathbb Z)$
such that $ \rho(g) A =A \rho(g)$ holds for all $g \in G$. Direct computation with the matrices from Proposition~\ref{monodromymatrices} shows
that our $\rho$ has the following endomorphism ring.
\begin{prop}
\label{deco-RM} % We have the decomposition
$\operatorname{End}(\rho) =\mathbb Z\operatorname{Id} \oplus \,\mathbb Z\operatorname{RM} \cong \mathbb{Z}[\sqrt{2}]$
where
\[
\operatorname{RM} := \begin{pmatrix}
      0&0&1&0\\
      0&0&0&2\\
      2&0&0&0\\
      0&1&0&0
   \end{pmatrix}
\]
is a matrix with
$\operatorname{RM}^2=2\operatorname{Id}$.
\end{prop}

\iffalse  % Moved before the proposition:
\begin{proof}
Direct verification with the explicit matrices in Proposition~\ref{monodromymatrices}.
\end{proof}
\fi

By taking bases of the $G$-invariant subspaces of $\mathbb{C}^4$ (or by taking eigenvectors of $\operatorname{RM}$)
we obtain the columns of the matrix
\[P:= \begin{pmatrix}
              0&\frac12\sqrt{2}&0&-\frac12 \sqrt{2}\\
              \sqrt{2}&0&-\sqrt{2}&0\\	
              0&1&0&1\\
              1&0&1&0
   \end{pmatrix}
\]
Consequently, the matrices  $P^{-1} M_i P$ have block diagonal form
\begin{equation} P^{-1} M_i P  = \begin{pmatrix} A_i& 0\\0&\tilde A_i \end{pmatrix}.  \label{splitting} \end{equation}
We find
\iffalse
	TO DO (MAYBE, OR MAYBE NOT): REPLACE THESE WITH THEIR INVERSES.  ALSO ADJUST BRAID BELOW BECAUSE THE ORDER IS NOW CHANGED FROM 4321 TO 1234.
	CURRENTLY WAITING FOR E-MAIL REPLY.
	I swapped A1, A3, that should do it.
\fi
\begin{alignat*}{2}
A_1&:= \begin{pmatrix} 2-\sqrt{2} &-\frac12 \sqrt{2} +1\\-2+\sqrt{2}&\sqrt{2} \end{pmatrix}, &\quad
A_2&:= \begin{pmatrix} 1 &0\\-2&1 \end{pmatrix}, \\                     
A_3&:= \begin{pmatrix} -\sqrt{2} &\frac12 \sqrt{2} +1\\-2-\sqrt{2}&2+\sqrt{2} \end{pmatrix}, &\quad
A_4&:= \begin{pmatrix} -1 &-1\\0&-1 \end{pmatrix}.
\end{alignat*}
The matrices $\tilde A_i$ are obtained from
$A_i$ with the Galois automorphism $\sqrt{2} \mapsto -\sqrt{2}$.
A splitting like~(\ref{splitting}) % of the monodromy representation
is equivalent
	% via the Riemann--Hilbert correspondence,	-->  NO,  it's because of differential Galois theory.   RH is about existence of differential equations
to the decomposition of $L^{\text{PF}}$ as the LCLM of two second order equations.
It is interesting that the real multiplication of the curves is visible in  ${\rm End}(\rho)$. % the eigenring of $L^{\text{PF}}$.
%  Text said "striking".  But is it really?

\begin{theorem}\label{monodromy-dense}
 The monodromy group of $L_{\pm}^{\textup{PF}}$ is a \emph{dense} subgroup of  $\operatorname{SL}_2(\mathbb{R})$.
\end{theorem}

\begin{proof}
The group $\operatorname{SL}_2(\mathbb{R})$ is generated by
\[
U := \bigg\{\begin{pmatrix} 1 & r \\ 0 & 1 \end{pmatrix} : r \in \mathbb{R}\bigg\}
\quad\text{and}\quad
L := \bigg\{\begin{pmatrix} 1 & 0 \\ r & 1 \end{pmatrix} : r \in \mathbb{R}\bigg\}.
\]
A computer search for ``small'' elements of the monodromy group $\langle A_1,A_2,A_3 \rangle$ produced
\[
u = \begin{pmatrix} 1 & 2 \sqrt{2} \\ 0 & 1 \end{pmatrix}
\quad\text{and}\quad
l = \begin{pmatrix} 1 & 0 \\ 4 \sqrt{2} & 1 \end{pmatrix}.
\]
Now $\langle u, A_4^2 \rangle$ is a dense subgroup of~$U$,
while $\langle l, A_2 \rangle$ is a dense subgroup of~$L$.
\end{proof}

\begin{corollary}\label{non-2F1}
If $L_{\pm}^{\textup{PF}}$ is solvable in terms of a $\mbox{}_2F_1$ function, then the monodromy group of the Euler--Gauss hypergeometric equation for this $\mbox{}_2F_1$ must
have a finite-index subgroup that is defined over $\mathbb{Q}(\sqrt{2})$ and is dense in $\operatorname{SL}_2(\mathbb{R})$.
\end{corollary} 

Our goal is to establish that $L_{\pm}^{\textup{PF}}$ is not $\mbox{}_2F_1$-solvable, and it appears that
we are \emph{almost} done.  All we have to do is to classify the $\mbox{}_2F_1(a, b ; c \mid x)$ that satisfy the assumption in Corollary~\ref{non-2F1}, and show that the list is empty.
However, it is not!  The monodromy of $\mbox{}_2F_1(\frac14, \frac58 ; 1 \mid x)$ satisfies this assumption.
And so the saga continues.

\subsection{Algebraic automorphisms}
\label{alg-auto}
How can we rule out % that $L_{\pm}^{\textup{PF}}$  is solvable
a solution in terms of   $\mbox{}_2F_1(\frac14, \frac58 ; 1 \mid f)$
for some algebraic function $f$? Call $d := [\mathbb{C}(f,x) : \mathbb{C}(x)]$ the \emph{algebraic degree} of~$f$.
%
%	Examples of this sort are known and normally come from modular parametrizations;
%
% Wadim, the phrase "of this sort" refers to something that has not yet been mentioned.
%
An example with $d = 2$ that comes from modular forms appears in the $\mbox{}_2F_1$-expression for the generating function of the Ap\'ery sequence.
The hypergeometric series $\mbox{}_2F_1(\frac14, \frac58 ; 1 \mid x)$ is of a different nature,
its Schwarz triangle has angles
(after scaling down by~$\pi$) $|1-c|=0$, $|c-a-b|=\frac18$ and $|b-a|=\frac38$, so it is not even on Takeuchi's list of arithmetic triangle groups \cite{Ta77}.
However, there do exist examples with $d > 1$ that do not come from modular forms, for instance:
% One example
% is the generating function for the Ap\'ery sequence, its generating function in OEIS sequence A005259
% is written in terms of $\mbox{}_2F_1(1/3, 2/3 ; 1 \mid f)$ with $f \in \mathbb{Q}(x, \sqrt{x^2-34x+1}$. More complicated
% examples exist as well:
\begin{equation} 
	%    2*x*(25*x^2+14*x+25)*Dx^2+(75*x^2+28*x+25)*Dx+3*x+1/2
	%    KEEP so we can copy/paste into Maple
	2x(25x^2+14x+25)y'' + (75x^2+28x+25) y' + (3x+ \tfrac12)y = 0.
\label{favourite_example}
\end{equation} 
Here $d=2$, which means that it is not $\mbox{}_2F_1$-solvable with a rational pullback, but it
does have a $\mbox{}_2F_1$-solution with a pullback $f$ in a quadratic extension. The implementation from \cite{IvH17}
finds a solution with $f \in \mathbb{Q}(x, \sqrt{25x^3+14 x^2+25x})$.  % , a rational pullback is not feasible for this example.

With our techniques for finding $\mbox{}_2F_1$-type solutions we rule out $d=1$ for $L_{\pm}^{\textup{PF}}$,
and $d=2$; the main difficulty is excluding a $\mbox{}_2F_1$-type solution without knowing a bound on~$d$.
% There is however a strategy, which we outline below, to reduce a potential $\mbox{}_2F_1$-solvability to the case $d=1$.
%    --> stated like this it's not clear if this refers to any 2F1 or just a particular one.
We will outline a strategy to potentially reduce to the case $d=1$.

If a second order differential equation can be solved in terms of $\mbox{}_2F_1(a, b ; c \mid f)$ for two different
pullback functions $f$, then we say that $\mbox{}_2F_1(a, b ; c \mid f)$ has \emph{non-unique pullbacks}.
If the equation has rational function coefficients but the algebraic degree is not~1, then taking two conjugates $f_1 \neq f_2$
of $f$ means that the $\mbox{}_2F_1$ has non-unique pullbacks.
If $g \in \overline{ \mathbb{C}(x) }$ is the composition of $f_1$ with the inverse of $f_2$, then 
then for a suitable algebraic pre-factor $r$,   the functions $\mbox{}_2F_1(a, b ; c \mid x)$
and $r \cdot \mbox{}_2F_1(a, b ; c \mid g)$ will satisfy the same Euler--Gauss hypergeometric equation.
We call such $g \in \overline{ \mathbb{C}(x) }$ with $g \neq x$ an {\em algebraic automorphism}. % of this $\mbox{}_2F_1$.
Non-unique pullbacks $f_1 \ne f_2$ can occur even in non-modular and non-arithmetic situations, like
\begin{gather*}
% \\
\mbox{}_2F_1\bigg(\begin{matrix} \frac1{12}, \, \frac16 \\ \frac12 \end{matrix}\biggm| f_1(x) \bigg)
= (1+x)^{-1/2} \cdot \mbox{}_2F_1\bigg(\begin{matrix} \frac1{12}, \, \frac16 \\ \frac12 \end{matrix}\biggm| f_2(x) \bigg),
\\
\text{with}\quad
f_1(x) = -x(1+x)(3+4x+4x^2)^2
\;\;\text{and}\;\;
f_2(x) = f_1\bigg(\frac{-x}{1+x}\bigg) = \frac{x(3+2x+3x^2)^2}{(1+x)^6},
\end{gather*}
and
\begin{gather*}
\mbox{}_2F_1\bigg(\begin{matrix} \frac1{40}, \, \frac{11}{40} \\ \frac45 \end{matrix}\biggm| f_1(x) \bigg)
=\bigg(1-\frac{44x}{25}\bigg)^{-1/10}\cdot\mbox{}_2F_1\bigg(\begin{matrix} \frac1{40}, \, \frac{11}{40} \\ \frac45 \end{matrix}\biggm| f_2(x) \bigg),
\\
\text{with}\quad
f_1(x) = \frac{256x(1-x)^5}{25-44x+20x^2}
\\\ \text{and}\;\;
f_2(x) = f_1\bigg(\frac{-x}{1-\frac{44}{25}x}\bigg) = -\frac{256x(25-19x)^5}{25(25-44x)^4(25-44x+20x^2)}.
\end{gather*}
Here $f_1, f_2$ are rational, but algebraic automorphisms do not always correspond to genus~0 extensions of $\mathbb C(x)$, e.g., in the case of~\eqref{favourite_example}.
So far all non-trivial cases%
\footnote{If two of the exponent-differences $|1-c|$, $|c-a-b|$, $|b-a|$ are equal, then a M\"obius transformation that swaps the two points trivially gives an ``algebraic'' automorphism $g \in \mathbb{Q}(x)$.}
correspond to Schwarz triangles with all angles either 0 and/or reciprocals of a positive integer, which does not include $\mbox{}_2F_1(\frac14, \frac58 ; 1 \mid x)$.

Another way to potentially prove that  $\mbox{}_2F_1(\frac14, \frac58 ; 1 \mid x)$ has no algebraic automorphisms
is by showing that a series expansion of any candidate $g$ has infinitely many primes in its denominators, which implies that $g$ is not algebraic.
There is a formula for the power series $g$ in terms of its first term \cite{vH13}.
The formula includes integration, which divides the coefficient of $x^{n}$ by $n+1$. When $n+1$ is prime $p$,
the division introduces $p$ in the denominator, unless a certain congruence $c_p \equiv 1 \bmod p$ holds, where $c_p$ is defined below.
The algebraicity of $g$ then implies that congruence for all sufficiently large primes $p$.
Thus, if we could prove that the congruence does not hold for infinitely many primes, it contradicts $g$ being an {\em algebraic} automorphism.
	% Remarkably enough, the underlying structure for the congruences seems to be simply controlled as we have the following expectation (which we do not know how to prove):
	%
	% --> Not sure how to interpret the word "expectation" in the above line.  The observation that the congruence holds for precisely
	%     those primes was for definitely unexpected
Computing $c_p$ for a number of primes leads to the following conjecture.
\begin{conjecture}
\label{conj1}
Define
\[
	Y(x) = \frac{9}{16} (1-x)^{7/8} \cdot \mbox{}_2F_1\bigg(\begin{matrix} \frac14, \, \frac58 \\ 1 \end{matrix}\Bigm| x\bigg)^2,
\]
and let $c_n$ be the coefficient of $x^n$ in $1/Y(x)$.  Let $p > 3$ be a prime number.
Then
\[
	c_p \equiv 1 \bmod p  \iff p \equiv \pm 1 \bmod 8.
\]
In particular, $c_p \not\equiv 1 \bmod p$ for infinitely many primes.
\end{conjecture}
Repeating this for other $\mbox{}_2F_1$ functions leads to other similar conjectures, such as:

\begin{conjecture}[exponent-differences $0$, $1/2$, $1/k$]
\label{conj2}
For $k \ge 3$, define
\[
Y(x) = \frac{3k^2+4}{8k^2} (1-x)^{1/2} \cdot \mbox{}_2F_1\bigg(\frac14 - \frac1{2k}, \, \frac14 + \frac1{2k} ; 1 \biggm| x\bigg)^2.
\]
Let $c_n$ be the coefficient of $x^n$ in the series expansion of $1/Y(x)$ at $x=0$.
Then for all but finitely many primes~$p$,
\[
	c_p \equiv 1 \bmod p \iff p \equiv \pm 1 \bmod k.
\]
\end{conjecture}

Conjecture~\ref{conj2} is true for $k = 3, 4, 6$.
(For these $k$, the $\mbox{}_2F_1$ does have algebraic automorphisms, which implies the congruence for all but finitely many primes~$p$.)

To summarise: while we cannot prove Conjecture~\ref{conj1}, it is easy to test it for thousands of primes.
So it is exceedingly unlikely that $\mbox{}_2F_1(\frac14, \frac58 ; 1 \mid x)$ has an algebraic automorphism.
If this were proven, it would imply unique\,---\,hence rational\,---\,pullbacks, allowing us to conclude that $L_{\pm}^{\textup{PF}}$ is not $\mbox{}_2F_1$-solvable.

Finally, note that a description (and understanding) of all $_2F_1$ cases with algebraic automorphisms is an interesting question itself.

\subsection{Braid action}
\label{braid}
We computed the monodromy of $L^{\text{PF}}$ in Section~\ref{monodromy} for one value of $u$, namely $u=1/2$. As long as the paths in diagrams do not change,
the monodromy matrices $M_1,\ldots,M_4$ in Proposition~\ref{monodromymatrices}
depend continuously on $u$. But they have integer entries, so they must be {\em locally constant}.
The monodromy matrices $A_1,\ldots,A_4$ of $L_{+}^{\textup{PF}}$ were derived from $M_1,\ldots,M_4$ in section~\ref{mono-RM}, so they
must be locally constant as well.

However, a sufficiently large change in $u$ does alter paths.
Figure~\ref{Pos1} shows what can happen when $u$ moves away from $1/2$, with $x_1$ and $x_3$ moving as a result, while 
$x_2 = 0$ and $x_4 = \infty$ are fixed.

For $M_i$ to depend continuously on $u$ (and thus remain constant) the path $\ell_i$ around $x_i$ must be dragged along.
Figure~\ref{Pos1} shows that when $u$ moves far enough, a dragged path need no longer be the shortest path around $x_i$.
If we recompute the $M_i$'s using the shortest paths $\ell_i^{\text{new}}$, and compare with the original $u=1/2$ monodromy matrices belonging to the dragged paths $\ell_i$, then the difference corresponds to a {\em braid action}.
Figure~\ref{Pos1} shows how the monodromy over $\ell_1$, $\ell_2$, $\ell_3$
	% (and $\ell_4$)  --> NOT shown in the figure (this is also mentioned later)
determines the monodromy over the new paths $\ell_1^{\text{new}}$, $\ell_2^{\text{new}}$, $\ell_3^{\text{new}}$.
	% (and $\ell_4^{\text{new}}=\ell_4$).
Since $\ell_1^{\text{new}} = \ell_2$, its matrix $M_1^{\text{new}}$ will be $M_2$.
In the figure, $\ell_3^{\text{new}} = \ell_3$ and $\ell_4^{\text{new}}=\ell_4$ (not drawn),
hence $M_3^{\text{new}}, M_4^{\text{new}}$ are the same as $M_3, M_4$. Then from~\eqref{product1} we find
\[
M_1 M_2 M_3 M_4 = \operatorname{Id} =
M_1^{\text{new}} M_2^{\text{new}} M_3^{\text{new}} M_4^{\text{new}} =
M_2 M_2^{\text{new}}  M_3 M_4 
\]
and hence $M_2^{\text{new}}=M_2^{-1} M_1M_2$.
This procedure gives a map
\[
	\Braid_{1,2}\colon  \operatorname{GL}_4(\C)^4   \to \operatorname{GL}_4(\C)^4
\]
which sends $(M_1, \ldots, M_4)$ to $(M_1^{\text{new}}, \ldots, M_4^{\text{new}}) = (M_2,\,M_2^{-1} M_1 M_2,\,M_3, M_4)$.
% We omitted $M_4$ as it can be obtained from equation~\eqref{product1}. 
We define $\Braid_{2,3}$ in the same way.  As for $\Braid_{3,4}$, the conjugacy classes of $M_3$ and $M_4$ differ by a minus sign,
so we compose $\Braid_{3,4}$ with multiplying $M_3$ and $M_4$ by~$-1$.  (Alternatively, one could only use {\em pure braids},
which map to the trivial permutation of $\{1,2,3,4\}$, such as $\Braid_{3,4}^2$.)

\begin{figure}[ht]
\begin{center}
\begin{tikzpicture}
\def\sunit{1.8}
\fill (-1*\sunit,1*\sunit) circle (2.1pt); \fill (-0.3*\sunit,-0.6*\sunit) circle (2.1pt); \node [above] at (-0.285*\sunit,-0.6*\sunit) {$x_1$};
\fill (0,0) circle (2.1pt); \node [right] at (0,0) {$x_2$};
\fill (-1*\sunit,-1*\sunit) circle (2.1pt); \fill (-1*\sunit,-2*\sunit) circle (2.1pt); \node [right] at (-1*\sunit,-1.9*\sunit) {$x_3$};
\fill (1.5*\sunit,0) circle (2.1pt); \node [right] at (1.5*\sunit,0) {$x_{\text{BP}}$};
\draw [dotted] (-1*\sunit,1*\sunit) .. controls (-0.9*\sunit,0) .. (-0.3*\sunit,-0.6*\sunit);
\draw [dotted] (-1*\sunit,-1*\sunit) -- (-1*\sunit,-2*\sunit);
\draw[decoration={markings, mark=at position 0.3 with {\arrow[line width=1.2pt]{>}}}, postaction={decorate}]
 (1.5*\sunit,0) .. controls (0.85*\sunit,0.25*\sunit) .. (0.2*\sunit,0.4*\sunit) .. controls (-0.6*\sunit,0.5*\sunit) and  (-0.75*\sunit,0.3*\sunit) .. (-0.45*\sunit,-0.6*\sunit) .. controls (-0.3*\sunit,-0.85*\sunit) and (-0.05*\sunit,-0.85*\sunit) .. (-0.12*\sunit,-0.55*\sunit) .. controls (-0.4*\sunit,0.0*\sunit) and (-0.4*\sunit,0.25*\sunit) .. (0*\sunit,0.25*\sunit) .. controls (0.7*\sunit,0.15*\sunit)  .. (1.5*\sunit,0);
\draw[decoration={markings, mark=at position 0.7 with {\arrow[line width=1.2pt]{>}}}, postaction={decorate}]
 (1.5*\sunit,0) .. controls (-0.1*\sunit,0.2*\sunit) .. (-0.15*\sunit,0) .. controls (-0.1*\sunit,-0.2*\sunit) .. (1.5*\sunit,0);
\draw[decoration={markings, mark=at position 0.3 with {\arrow[line width=1.2pt]{>}}}, postaction={decorate}]
 (1.5*\sunit,0) .. controls (-1.25*\sunit,-1.8*\sunit) .. (-1.2*\sunit,-2.1*\sunit) .. controls (-1.0*\sunit,-2.35*\sunit) .. (1.5*\sunit,0);
\node [above] at (0.55*\sunit,0.3*\sunit) {$\ell_1$};
\node [below] at (0.3*\sunit,-0.1*\sunit) {$\ell_2$};
\node [below] at (0.4*\sunit,-1.1*\sunit) {$\ell_3$};
\end{tikzpicture}
\qquad
\begin{tikzpicture}
\def\sunit{1.8}
\fill (-1*\sunit,1*\sunit) circle (2.1pt); \fill (-0.3*\sunit,-0.6*\sunit) circle (2.1pt); \node [right] at (-0.3*\sunit,-0.55*\sunit) {$x_1$};
\fill (0,0) circle (2.1pt); \node [right] at (0,0) {$x_2$};
\fill (-1*\sunit,-1*\sunit) circle (2.1pt); \fill (-1*\sunit,-2*\sunit) circle (2.1pt); \node [right] at (-1*\sunit,-1.9*\sunit) {$x_3$};
\fill (1.5*\sunit,0) circle (2.1pt); \node [right] at (1.5*\sunit,0) {$x_{\text{BP}}$};
\draw [dotted] (-1*\sunit,1*\sunit) .. controls (-0.9*\sunit,0) .. (-0.3*\sunit,-0.6*\sunit);
\draw [dotted] (-1*\sunit,-1*\sunit) -- (-1*\sunit,-2*\sunit);
\draw[decoration={markings, mark=at position 0.3 with {\arrow[line width=1.2pt]{>}}}, postaction={decorate}]
 (1.5*\sunit,0) .. controls (-0.1*\sunit,0.2*\sunit) .. (-0.15*\sunit,0) .. controls (-0.1*\sunit,-0.2*\sunit) .. (1.5*\sunit,0);
\draw[decoration={markings, mark=at position 0.35 with {\arrow[line width=1.2pt]{>}}}, postaction={decorate}]
 (1.5*\sunit,0) .. controls (-0.5*\sunit,-0.4*\sunit) .. (-0.5*\sunit,-0.7*\sunit) .. controls (-0.45*\sunit,-0.85*\sunit) .. (1.5*\sunit,0);
\draw[decoration={markings, mark=at position 0.3 with {\arrow[line width=1.2pt]{>}}}, postaction={decorate}]
 (1.5*\sunit,0) .. controls (-1.25*\sunit,-1.8*\sunit) .. (-1.2*\sunit,-2.1*\sunit) .. controls (-1.0*\sunit,-2.35*\sunit) .. (1.5*\sunit,0);
\node [above] at (0.3*\sunit,0.1*\sunit) {$\ell_1^{\text{new}}$};
\node [left] at (-0.45*\sunit,-0.7*\sunit) {$\ell_2^{\text{new}}$};
\node [below] at (0.4*\sunit,-1.1*\sunit) {$\ell_3^{\text{new}}$};
\end{tikzpicture}
\end{center}
\caption{Positioning after a $u$-move}
\label{Pos1}
\end{figure}
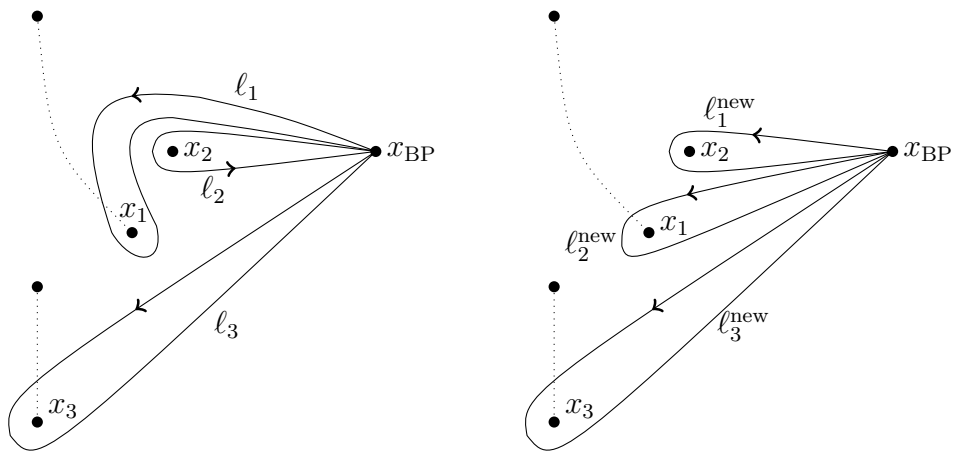

To obtain independence of basis changes, let $\sim$ be the equivalence relation on  $\operatorname{GL}_4(\C)^4$
where $(M_1, \ldots, M_4) \sim (\tilde{M}_1, \ldots, \tilde{M}_4)$ if they are simultaneously conjugated (that is, $\tilde{M}_i = P^{-1} M_i P$ for some $P$ and all~$i$).
We denote the equivalence class of $(M_1, \ldots, M_4)$ as
\[ [M_1,\ldots, M_4]_\sim \in  \operatorname{GL}_4(\C)^4 /\hspace{-3pt}\sim. \]
The braid group $\langle \Braid_{1,2}, \Braid_{2,3}, \Braid_{3,4} \rangle$ acts on $\operatorname{GL}_4(\C)^4/\hspace{-3pt}\sim$ well.
We denote the orbit of $[M_1,\ldots,M_4]_\sim$ as ${O}([M_1,\ldots,M_4]_\sim) \subset \operatorname{GL}_4(\C)^4/\hspace{-3pt}\sim$.

Like $L^{\text{PF}}$,  the monodromy $(A_1,\ldots,A_4)$  of  $L_+^{\text{PF}}$ is locally constant, and hence,
changing the value its parameter $t$ % in $L_+^{\text{PF}}$
% moves some singularities of $L_+^{\text{PF}}$,
causes braid actions. % on its monodromy.
The braid orbit contains all monodromies that can be obtained from $L_+^{\text{PF}}$ for non-degenerate values of $t$.
The simplest monodromy we found in the braid orbit was
\begin{equation}
\left[
\begin{pmatrix}
1 & 1 
\\
 0 & 1 
\end{pmatrix}, \;
\begin{pmatrix}
1 & 0 
\\
\sqrt{2}-2 & 1 
\end{pmatrix}, \;
\begin{pmatrix}
1 & 0 
\\
-\sqrt{2}-2 & 1 
\end{pmatrix}, \;
\begin{pmatrix}
1 & -1 
\\
4 & -3
\end{pmatrix}\right]_\sim.	\label{O8}
\end{equation}
% (The last matrix was omitted, it is the inverse of the product of the other matrices).
The following can be shown in a way similar to Section~5 in \cite{vHK19}.

\begin{remark} \label{Rem2}
Let $L$ be a second order regular singular differential equation, with four true singularities, and an
arbitrary number of apparent singularities.  If the monodromy matrices of $L$ are in the above
mentioned braid orbit, then $L$ can be solved in terms of solutions of $L_+^{\text{PF}}$ for some $t$.
In particular, any such $L$ must be arithmetic and have a solution that can be written in terms of a hyperelliptic integral.
\end{remark}

\begin{remark} \label{BraidO}
Let $\zeta$ be a root of unity and let $\alpha = \zeta + \zeta^{-1} \in \mathbb{R}$.
Let $O_{\zeta}$ be the braid orbit of the four matrices in $\operatorname{GL}_2( \mathbb{Z}[\alpha] )$ obtained by replacing  $\sqrt{2}$ in~\eqref{O8} with $\alpha$.

By the Riemann--Hilbert correspondence there
should exist a family of second order differential equations with monodromy in $O_{\zeta}$.
Three singularities $x_1,x_2,x_3$ can be moved to $0,1,\infty$ with M\"obius transformations, leaving one singularity $x_4$ to move freely. 
The braid orbit $O_{\zeta}$ is {\em finite}, so we expect this family to be definable over a finite extension of $\mathbb{Q}(x_4)$.
In other words, we expect an algebraically defined one-dimensional family of differential equations for $O_{\zeta}$,
just like $L_+^{\text{PF}}$ for $O_{\zeta_8}$.
\end{remark}

\begin{remark} \label{Painleve}
As pointed out by the anonymous referee, certain aspects of our results in a different context appear in the literature.
Recall that the Painlev\'e VI equation $P_{\text{VI}}$ characterizes the isomonodromic deformation equation of second order Fuchsian differential equations with \emph{four} regular singular points and an apparent singular point on the Riemann sphere.
If the monodromy tuple has a finite orbit, then the corresponding Painlev\'e VI equation has algebraic solutions; all such tuples are classified in~\cite{LT14}. The differential equation $L_\pm^{\text{PF}}$ corresponds to the so-called Picard case $P_{\text{VI}}(0, 0, 0, 1)$, see~\cite{Ma01}; an Okamoto transformation (see, for example, \cite[Appendix~B]{CL07}) brings the latter to the equation $P_{\text{VI}}(1/2, 1/2, 1/2, 1/2)$ that corresponds to a monodromy tuple with four reflections and the same action
of the braid group \cite[Theorem~2.2 and Example~6.5]{DR07}.
These four reflections generate an orthogonal subgroup of $\operatorname{GO}_2(\mathbb C)$, as taking the second symmetric power yields a fixed space.
This group is finite if and only if the braid group orbit is finite.
Here the generated group is the dihedral group $D_8$ (of size~16) with braid orbit length~8.
The Okamoto transformation above can be seen as a \emph{middle convolution} with $x^{1/2}$ (compare with \cite[Section~6]{DR07}).
Applying this middle convolution $MC_{-1}$ to a symplectic monodromy tuple yields an orthogonal one; this explains why the Hadamard product of the periods with $(1-4x)^{1/2}$ is an algebraic function, which in turn, by a known result of Klein, is a pullback of a hypergeometric
function.
%This should also answer Remark \cite{BraidO}.
Furthermore, the middle convolution is invertible and can be written explicitly for the monodromy group generators.
This gives one a sanity check that the generators of the monodromy tuple for $L_\pm^{\text{PF}}$ are as claimed (up to braiding).
\end{remark}

\section{Around Dwork's conjecture}
\label{dwork}

\subsection{Dwork's conjecture}
\label{subdwork}
A conjecture of Dwork \cite[p.~784]{Dw90} from 1990 suggests that \emph{globally nilpotent} (over a number field) second order differential equations are $_2F_1$-solvable; see \cite[Section~7]{Dw90} for a precise statement.
In this form, the conjecture was shown to be false by Krammer~\cite{Kr96} and later by Dettweiler, Reiter \cite{DR10}.
Bouw and M\"oller \cite{BM10} (see also \cite{MZ16}) gave examples of \emph{globally bounded} but not $_2F_1$-solvable equations over the orders of some real quadratic extensions of~$\mathbb{Q}$.
%These examples disprove Dwork's conjecture for {globally bounded} second order differential equations over number fields.
Solving numerous second order differential equations obtained from the OEIS \cite{OEIS} suggested that this refined version of Dwork's conjecture
might still be true over~$\mathbb Q$\,---\,this was explicitly stated as Conjecture~1 in \cite{vHK19}. However, Corollary~\ref{non-2F1}
combined with Section~\ref{alg-auto} strongly indicate that this is not the case.

All examples prior to our work have been Heun-solvable, that is, they are algebraic pullbacks of suitable Heun differential equations (with exactly four singularities).
Our equation $L_+^x$ has an additional apparent singularity, so it is not of Heun type.
To give a more detailed comparison, we took an example from \cite{MZ16}, wrote it as a pullback of the Heun equation:
\begin{align*}
L_{\text{Heun}}
&=\frac{\d^2}{\d t^2}
+\frac{768t^2+2(638-217\sqrt{17})t-895+217\sqrt{17}}{t(t-1)(256t+895-217\sqrt{17})}\, \frac{\d}{\d t}
\\ &\quad
+\frac{60(4t+3-\sqrt{17})}{t(t-1)(256t+895-217\sqrt{17})}
\end{align*}
and computed the monodromy of the latter. Up to equivalence, it is
\[
\left[
\begin{pmatrix} 1 & 1 \\  0 & 1 \end{pmatrix}, \;
\begin{pmatrix} 1 & 0 \\ -\frac{7}{2}-\frac{\sqrt{17}}{2} & 1 \end{pmatrix}, \;
\begin{pmatrix} -\frac{3}{2}-\frac{\sqrt{17}}{2} & \frac{9}{8}+\frac{\sqrt{17}}{8} \\ -\frac{13}{2}-\frac{3 \sqrt{17}}{2} & \frac{7}{2}+\frac{\sqrt{17}}{2} \end{pmatrix}
\right]_\sim
\]
(the omitted 4th matrix is the inverse product of the other matrices).
The braid orbit of this monodromy is not finite.  That distinguishes it from our example $L_+^x$, which (recall Remark~\ref{BraidO}) is an example of the following type of family of differential equations:
four non-apparent singularities that are moved non-trivially by an algebraic parameter $t$,
where the monodromy is {\em locally constant} when we vary $t$ (avoiding degenerate $t$-values where singularities collide).
Any member of this type of family (obtained by substituting a non-degenerate value for~$t$) will necessarily have a monodromy with a finite braid orbit.
Thus $L_{\text{Heun}}$, with its infinite braid orbit, does not occur as a member {\em in this type of family}.
Several features\,---\,defined over~$\mathbb{Q}$, and an additional parameter\,---\,appear to be novel in our (expected to be a) counter example to Dwork's conjecture for the globally bounded case.

Having discussed differences,  similarities are the subject of Section~\ref{teich} below.
One starts with a Hilbert modular surface and a family of hyperelliptic curves on it.

\subsection{Teichm\"uller curves on Hilbert modular surfaces}
\label{teich}
For an order $\mathcal{O}_D$ in a real quadric field of discriminant $D$, where $D\equiv0$ or $1\bmod4$ is not a square, one forms the Hilbert modular surface
\[ X_D:=\mathbb H^2/\operatorname{SL}_2(\mathcal{O}_D). \]
It can be interpreted as the moduli space of special principally polarised abelian surfaces, and the
image of the corresponding map $X_D \to \mathcal{A}_2$ is the Humbert surface $H_D$.
McMullen has shown (see \cite{Mc23} for an overview) that on each $X_D$ there is one (primitive) Teichm\"uller curve, unless $D \equiv 1\bmod 8$ in which case there are two such curves.
It is known that the Picard--Fuchs equation for the corresponding family of $g=2$ Veech curves
over such a Teichm\"uller curve decomposes, one factor being the uniformising differential equation for
the (open) Teichm\"uller curve.  A salient feature of these Teichm\"uller curves is that
the eigenform has a double zero on the hyperelliptic curve. The monodromy groups are discrete in
$\operatorname{SL}_2(\mathbb R)$, but not commensurable to the modular group $\operatorname{SL}_2(\Z)$.
The example of Bouw and M\"oller corresponds to the case $D=17$ and their operator is defined over $\Q(\sqrt{17})$.

In our situation, we are dealing with a two-parameter family of hyperelliptic curves with real multiplication
by $\sqrt{2}$. This means that our family is pulled back from the Hilbert modular surface $X_8$.
For a general member $C^{u,x}$ of our family, the differential does not have a double zero.
This shows that the phenomena described by Bouw and M\"oller are not due to the Teichm\"uller property,
but seem to stem from the real multiplication itself and
	% consequently			% I'm not sure I understand this
extend %, properly interpreted,		%
to the
whole surface
	% $X_D$.					% If we don't explain this for arbitrary D,  then the next best thing is to show it for two D's,  D=5 and D=8.  Hence I'll pick one here:
$X_8$. We now describe this
	% calculation that fortunately happens to be rather simple. For this, let
double zero scenario for $X_8$.  Let $\omega$ be the differential in equation~\eqref{hyperY}, which is then of the form
\[
\frac{\sqrt{-2}}{\pi}\int \omega
\]
This $\omega$ is a differential of the first kind on our hyperelliptic curve, defined by the equation $Y^2=H(v)$, where
$H(v)$ denotes the sextic in $v$ (depending also on $t$ and $x$) from the square root of the denominator.
In general, such a differential has $2g-2=2$ zeros, and they appear at the two points where the numerator of the differential $\omega$ vanishes: that is, when $v=1-t$, hence $Y=\pm\sqrt{H(1-t)}$. By the theory of McMullen (see \cite{Mc23} for an overview) the case, where these two points coalesce, characterizes the Teichm\"uller condition inside the Humbert (or Hilbert) modular surface.
It is clear that these two points can only coalesce when they become a root of the defining sextic.
Therefore we are left to study the curve in the $(t,x)$-space defined by the condition $H(1-t)=0$, which after the expansion reads
$(t-1)(1 - 2t - 128t(1-t)^3 x)$.
Ignoring the trivial double root originating from $t=1$ we take a closer look at the one coming from
\[
x=\frac{1-2t}{128t(1-t)^3}.	% Replaced 32 with 128
\]
Substituting this into the integral \eqref{hyperY} gives
\begin{equation*}
I = \frac{2}{\pi}\int  \frac{(t-1)^{3/2} (v+t-1) \ \d v \ \,}{\sqrt{
(v+t)(tv+t^2-1)(v-t+1)(v+t-1)(v^2(1-2t)+2t^3-3t^2+1)
}}.
\end{equation*}
% Notice that now $t$ is the parameter, rather then our earlier $x$.
This integral satisfies % the differential operator
\[
\frac{\d^2}{\d t^2} I + \frac{2(8t^4-16t^3+6t-1)}{ (t-1)(4t^2-1)(2t^2-1)} \frac{\d}{\d t} I
+\frac{ 2t^4-4t^3+15t^2-6t-1}{(t-1)^2(4t^2-1)(2t^2-1)} I = 0.
\]
Solving that gives % the corresponding differential equation gives
\[
\frac{1-t}{t^{3/2}}
 \, {}_2F_1\biggl(\begin{matrix} \frac38, \, \frac38 \\ 1 \end{matrix}\biggm|\frac{4t^2-1}{4 t^4}\biggr)
\]
which equals $I$ if the integral is taken from $v  = -t$ to  $v=t^{-1}-t$ and $t$ is sufficiently large.
The monodromy group of this ${}_2F_1$ is the (non-arithmetic!) triangle group $(4,\infty,\infty)$.
Thus, our curve maps to the Teichm\"uller curve $W_8$.

\begin{remark}[The case $X_5$]  % I deleted "an example" because from while X_5 may be an example, we fully cover the corresponding monodromy
\label{Remark2}
The relation between $X_8$ and $L_+^x$ motivated us to compute a similar family of differential equations for the surface $X_5$. After some
transformations to reduce its size we found
\begin{align*}
	L_5 &= \frac{\d^2}{\d x^2}  + \frac{2x^3+(u^2+6u+1)x^2 - 16u^2}{x (x^3+(u^2+6u+1)x^2+8u(u+1)x+16u^2)} \,\frac{\d}{\d x}
\\ &\qquad
	+\frac{(25x^2-(u+3)^2 x-40u(u+3))}{100x(x^3+(u^2+6u+1)x^2+8u(u+1)x+16u^2)}.
\end{align*}
The monodromy of $L_5$, for any non-degenerate value of $u$, is in $O_{\zeta_5} = O_{\zeta_{10}}$ from Remark~\ref{BraidO}.
A claim like Remark~\ref{Rem2} applies here as well, every equation whose monodromy is in $O_{\zeta_5}$
\iffalse
	the orbit of
	\[
	\left[
	\begin{pmatrix}
	1 & 1 
	\\
	 0 & 1 
	\end{pmatrix}, \;
	\begin{pmatrix}
	1 & 0 
	\\
	\varphi - 2 & 1
	\end{pmatrix}, \;
	\begin{pmatrix}
	1 & 0 
	\\
	-\varphi - 2 & 1
	\end{pmatrix}, \;
	\begin{pmatrix}
	1 & -1 
	\\
	4 & -3
	\end{pmatrix}\right]_\sim
	\]
\fi
can be solved in terms of $L_5$ and hence in terms of hyperelliptic integrals.
\iffalse
	Here $\varphi = (1+\sqrt{5})/2$ is the golden ratio.
\fi
Taking a local solution of $L_5$ at $x = \infty$, multiplying the local parameter by $5$ and replacing $u$ with $5u-3$,
produces a power series whose coefficients in $\mathbb{Z}[u]$ are divisible by $\binom{2n}{n}$ and satisfy the recursion
\begin{align*}
&
(n+1)^2 a_{n+1}
+5\big((5n+3)(5n+2)u^2-2(2n+1)^2\big) a_n
\\ &\;
+ 100(5u-3)\big(2(5u-2)n^2-3u+1\big) a_{n-1}
+500(5u-3)^2(2n+1)(2n-3)a_{n-2}
= 0
\end{align*}
for $n=0,1,2,\dots$, with the convention $a_0=1$ and $a_n=0$ for $n<0$.
\end{remark}

\section{Cubes of Legendre polynomials}\label{cube}

We now briefly turn our attention to the generating function \eqref{cub}.
% which gets a different treatment than~\eqref{orig}.

\begin{theorem}
\label{leg-cube}
Let $p = 1- x y^3 + x^2 (3 y^2 - 2)/4$ and
% \[ H(x) = p^{-1/2} \ \mbox{}_2F_1( 1/4, 3/4, 1, p^{-2} (y^2-1)^3 (x^2 - \frac14 x^4)  ) \]
\[ H = \frac1{\sqrt{p}} \cdot {}_2F_1\biggl(\begin{matrix} \frac14, \, \frac34 \\ 1 \end{matrix} \bigg|  \frac{(y^2-1)^3 (x^2 - \frac14 x^4)}{p^2} \biggr). \]
Then
\[ F_{P^3}(y,z) = \frac1{\sqrt{1+z^2}} \left( H \star \frac1{\sqrt{1-4x}} \right) \bigg|_{x = z/(1+z^2)}  \]
where the Hadamard product $\star$ is with respect to the internal variable~$x$.
\end{theorem}

Again it is relatively easy to verify this formula, and the proof below mainly outlines how it was found.

\begin{proof}
First observe $z \mapsto 1/z$ symmetry in the singularities.
That suggests that we can write the equation in terms of $x = z/(1+z^2)$, a generator of the subfield of $\mathbb Q(z)$, which is fixed under the symmetry.
Indeed, the corresponding change of variables % z -> RootOf( x = z/(1+z^2), z)
produces the equation that is nearly rational in~$x$ (it is rational after a simple multiplication of solutions by an algebraic factor).
Then unexpectedly, like in subsection~\ref{Had}, we observe that the new equation has the `central binomial twist' property: the coefficients of its analytic solution at the origin can be divided by $\binom{2n}{n}$ leading to a series which is still globally bounded.
More surprisingly, the differential equation for the resulting series has order~2!
A closed form (hypergeometric) solution can then be found quickly.
\end{proof}

The fourth order differential operator $L_{P^3}$ for \eqref{cub} with respect to the variable $z$ decomposes as a symmetric product of second order equations just as in Section~\ref{discovery}.
But the coefficients are in a genus~1 function field. The question remains if these second order equations can also be solved in terms of algebraic integrals.
We hoped to find a connection to some $O_{\zeta}$ from Remark~\ref{BraidO}, but those have 4 logarithmic singularities, a feature we did not encounter
in equations we were able to obtain from $L_{P^3}$.

\section{Conclusion and future work}
\label{final}

A task of recognising the innocent-looking generating function \eqref{orig} of the squares of Legendre polynomials in a closed form has turned out to produce families of hyperelliptic integrals that satisfy second order Picard--Fuchs differential equations, with the latter having remarkable features.
%%%
%%% EDIT THIS NEXT PART?
%%%
Next one can investigate other Humbert surfaces corresponding to different orders, and we did for $X_5$ in Remark~\ref{Remark2}.
The corresponding second order equations are solved by hyperelliptic integrals with the underlying Jacobians admitting RM by $\mathbb Z[(1+\sqrt5)/2]$. We plan to do this for more Humbert surfaces elsewhere;
the underlying geometric structures for small orders are already classified in \cite{EK16,KM16}.

%\begin{remark}[Algorithm for arbitrary $N$]
After analyzing the relation between $X_8$ and $O_{\zeta_8}$ from Remark~\ref{BraidO}, and $X_5$ and $O_{\zeta_5}$ in Remark~\ref{Remark2},
the first author constructed a program (available at \cite{supp})
that takes as input a positive integer $N$ and a non-cuspidal point $P$ on $X_1(N)$.
It computes a second order differential operator $L_{N,P}$, and a hyperelliptic integral that satisfies $L_{N,P}$,  both defined over $\mathbb{Q}(P)$.
To our surprise, not only can hyperelliptic integrals satisfy second order equations, but also, we now have a program that can produce examples of arbitrarily high genus.
This $L_{N,P}$ conjecturally has monodromy in $O_{\zeta_N}$, implying that the algorithm produces families of hyperelliptic curves
with real multiplication by $\mathbb{Z}[ \zeta_N + \zeta_N^{-1} ]$, as in Mestre \cite{Me91}.
	% explicit hyperelliptic curves defined over $\mathbb{Q}(P)$.
	% raises the question what to expect for other $X_n$
	% (the only remaining $n$ for which $\zeta_n + \zeta_n^{-1}$ from Remark~\ref{Remark2} is quadratic is $n=12$).
%\end{remark}

% One can also compare with~\cite{Me91}, which gives a family of curves of genus $(p-1)/2$ with $\zeta_p+ \zeta_p^{-1}$ in its endomorphism ring,
% or with items (6) and (7) in the main theorem in~\cite{El01}, which covers products of two odd primes, and odd prime powers.

Interestingly, the discussed surface $X_8$ (and it only!) has received a differential treatment 35~years ago in \cite{SY88}:
as an illustration of their general construction, Sasaki and Yoshida explicitly gave a rank~4 system of partial differential equations for the Hilbert surface $X_8$.
We expect that system to be related to the one from Section~\ref{discovery}. To test this we checked that it has a similar decomposition to rank~2, which it does.
\iffalse
However, they did not observe the system decomposing. We did not yet try to identify the corresponding rank~2 factors in the decomposition.
But it seems to be worth making the connection explicit,
between the differential equations $\operatorname{DE}_y$, $\operatorname{DE}_z$ and the Sasaki--Yoshida system.
\fi

Our proof of Theorem~\ref{th1} originates the following expectation of an arithmetic flavour, which may be regarded as a potential correction of Dwork's ex-conjecture.
Assume that $f(x)=\sum_{n=0}^\infty a_nx^n$ is in $1+x\mathbb{Z}[[x]]$ and that its twist $\sum_{n=0}^\infty\binom{2n}{n}a_nx^n$ by the central binomial coefficients satisfies a second order linear differential equation with regular singularities.
Must $f(x)$ be algebraic?

Our proofs of Theorems~\ref{th1} and~\ref{leg-cube} also demonstrate that closed forms of arithmetic (globally bounded) series may be obtained from `untwisting' the coefficients by the central binomial coefficients $\binom{2n}{n}$
% =2^{2n}(\frac12)_n/n!$
(or more generally by $(\alpha)_n/n!$ for some $\alpha \in \mathbb{Q}\setminus\mathbb Z$).
Such untwisting
	% is possible through a
may depend on a
special choice of the variable, with respect to which the series are expanded. The choices required in our proofs were \emph{ad hoc}, though hinted by the corresponding differential equations.
We would therefore question whether there exists a \emph{routine} strategy for performing the related change of variable directly from the equations, in a greater generality.
Such a methodology, if it exists, has potential to shed more light on a famous conjecture of Christol \cite{Ch15} about globally bounded solutions of Picard--Fuchs differential equations.

\medskip
\noindent
\textbf{Acknowledgements.}
We thank Alin Bostan for his comments on the proof of Theorem~\ref{th1} and Stefan Reiter for pointing out the reference~\cite{BR13}. Furthermore, we are thankful to the anonymous referee for their enthusiastic report and for bringing to us the connection with Painlev\'e equations.
%We are open to thank others.
The third author (W.Z.) thanks the Max Planck Institute for Mathematics in Bonn for its hospitality in July 2023 when part of his work on this project was materialized.

%==================================================

\end{document}